\documentclass[]{article}

\usepackage[utf8]{inputenc}

\usepackage{fullpage}

\usepackage{comment}
\usepackage{amsmath}
\usepackage{amssymb}
\usepackage{amsthm}
\usepackage{enumitem}
\usepackage{tikz}
\usetikzlibrary{positioning}
\usetikzlibrary{arrows,decorations.markings}

\usepackage{boxedminipage}

\usepackage{listings}

\usepackage{ifpdf}
\usepackage{hyperref}

\usepackage{todonotes}

\newtheorem{remark}{Remark}
\newtheorem{theorem}{Theorem}
\newtheorem{lemma}[theorem]{Lemma}
\newtheorem{example}[theorem]{Example}
\newtheorem{corollary}[theorem]{Corollary}

\newtheorem{proposition}[theorem]{Proposition}
\newtheorem{observation}[theorem]{Observation}
\newtheorem{claim}{Claim}

\newtheorem{conjecture}[theorem]{Conjecture}

\newcommand{\Z}{\mathbb Z}
\newcommand{\R}{\mathbb R}

\newcommand{\w}{{\mathbf w}}

\newcommand{\ba}{{\mathbf a}}
\newcommand{\bb}{{\mathbf b}}
\newcommand{\bc}{{\mathbf c}}

\def\a{\mathbf a}
\def\b{\mathbf b}
\def\c{\mathbf c}

\newcommand{\beeq}{\begin{eqnarray*}}
	\newcommand{\eneq}{\end{eqnarray*}}

\newcommand{\be}{\begin{equation}}
\newcommand{\ee}{\end{equation}}

\theoremstyle{remark}

\usepackage{xcolor}
\definecolor{darkred}{cmyk}{.3,.9,.80,.2}

\title{Extremal families  for the  Kruskal--Katona theorem}
\author{Oriol Serra\thanks{Department of Mathematics, Universitat Polit\`ecnica de Catalunya, Barcelona. \texttt{oriol.serra@upc.edu}. Supported by the Spanish Research Agency under project MTM2017-82166-P}
\and Llu\'{i}s Vena
\thanks{Department of Mathematics, Universitat Polit\`ecnica de Catalunya, Barcelona. \texttt{lluis.vena@gmail.com}. Supported by the a Beatriu de Pin\'os Fellowship BP2018-0030 of the AGAUR, Horizon's 2020 program cofund.}}

\date{\today}

\begin{document}

	\maketitle
		\begin{abstract}
Given a family  $S$ of $k$--subsets of  $[n]$, its lower shadow $\Delta(S)$ is the family of    $(k-1)$--subsets  which are contained in at least one set in $S$. The celebrated Kruskal--Katona theorem gives the minimum cardinality of $\Delta(S)$ in terms of the cardinality of $S$. F\"uredi and Griggs (and M\"ors) showed that the extremal families for this shadow minimization problem in the Boolean lattice are unique for some cardinalities and asked for a general characterization of these extremal families.

In this paper we prove a new combinatorial inequality from which yet another simple proof of the Kruskal--Katona theorem can be derived. The inequality can be used to obtain a characterization of the extremal families for this minimization problem, giving an answer to  the question of F\"uredi and Griggs. Some known and new additional properties of extremal families can also be easily derived from the inequality. 

	\end{abstract}

\tableofcontents

\section{Introduction}

The well--known Kruskal--Katona  Theorem \cite{kat09,krus63} on the minimum shadow of a family of $k$--subsets of   $[n]=\{1,2,\ldots ,n\}$ is a central result in Extremal Combinatorics with multiple applications, see e.g. \cite{Frankl1991}.
The \emph{shadow}  of a  family 
$S\subset \binom{[n]}{k}$  is the family
 $\Delta (S)\subset \binom{[n]}{k-1}$  of $(k-1)$--subsets  which are contained in some set in $S$. The Shadow Minimization Problem asks for the minimum cardinality of $\Delta (S)$ for families of $k$-sets $S$ with a given cardinality $m=|S|$. The answer given by the Kruskal--Katona theorem  can be stated in terms of  $k$--binomial decompositions. The \emph{ $k$--binomial decomposition} 
of a positive integer $m$ is 
\begin{equation}
\label{eq:k-bin-seq}
m=\binom{a_0}{k}+\binom{a_{1}}{k-1}+\cdots +\binom{a_{t}}{k-t},
\end{equation}
where the coefficients satisfy  $a_0>a_1>\cdots >a_{t}\ge k-t\ge 1$, $t\in [0,k-1]$, and are   uniquely determined by $m$ and $k$. We also refer to the  decreasing sequence $\a=(a_0>a_1>\cdots >a_{t})$ from above as the  $k$--binomial decomposition
 of $m$. More generally, for an integer sequence $\a=(a_0,a_1,\ldots ,a_t)$ we denote by  
 \[
\binom{\ba}{k}=\binom{a_0}{k}+\binom{a_{1}}{k-1}+\cdots + \binom{a_t}{k-t},
\]
and, for positive integers $i,j$ we also denote by 
\[
\binom{\ba-i}{k-j}=\binom{a_0-i}{k-j}+\binom{a_{1}-i}{k-j-1}+\cdots + \binom{a_t-i}{k-j-t}.
\]
We denote the length of the sequence $\a$ by $\ell (\a)=t+1$. 
If $\a$ is strictly decreasing, $\ell (\a)\le k$ and $a_t\geq k-t$ then  $m= \binom{\a}{k}$ is the 
\emph{$k$--binomial decomposition} of a strictly positive integer $m$;
in this case we use the shorthand $m\stackrel{b}{=}\binom{\a}{k}$. If 
$\a$ is  the empty sequence we set $\binom{\a}{k}=0$ by convention and write
 $0\stackrel{b}{=}\binom{\a}{k}$. 
 We use the usual conventions for binomial coefficients that 
 $\binom{n}{k}=0$ for all $k<0$ and $\binom{0}{0}=1$. 

For $i\le k-1$ the $i$--shadow $\Delta^i(S)$ of $S$ is defined recursively as $\Delta^i(S)=\Delta(\Delta^{i-1}(S))$ and $\Delta^0(S)=S$.  With this notation, the Kruskal--Katona theorem is expressed as follows.

\begin{theorem}[Kruskal--Katona \cite{zbMATH03489128,krus63}]\label{thm:kk}
	Let $n\geq k\geq 1$.
	Let $S\subset \binom{[n]}{k}$ be a non-empty family of $k$--subsets of $[n]$  and let  $\binom{\a}{k}$ be the $k$--binomial decomposition of $m=|S|$. Then, for each $1\le i\le k-1$,  
	\begin{equation}
	\label{eq:kk}
	|\Delta^i S|\ge \binom{\a}{k-i}.
	\end{equation}

\end{theorem}

The lower bound in Theorem \ref{thm:kk} is tight. We say that a family $S$ is \emph{extremal} 
if the cardinality of its lower shadow $\Delta (S)$ achieves the lower bound in Theorem \ref{thm:kk} for $i=1$.
It turns out that, for an extremal family, equality holds in \eqref{eq:kk} as well for all $i\ge 1$ (see \cite[Theorem~2.1]{furgri86}). In other words, the shadow of an extremal family is also  extremal.

We recall that the colex order 
on the $k$--subsets of $[n]$ is defined by $X\le_{\text{colex}} Y$ if and only if $\max((X\setminus Y)\cup (Y\setminus X))\in Y$. We denote by  $I_{n,k}(m)$   the initial segment  of length $m$ in the colex order in $\binom{[n]}{k}$.  In what follows, we refer to initial segments of the colex order up to automorphisms of the Boolean lattice, induced by any permutation of $[n]$. It can be easily checked that, for every $m$, the initial segment $I_{n,k}(m)$ of length $m$ in the colex order is an extremal family.

F\"uredi and Griggs \cite{furgri86} (see also M\"ors \cite{mors85}) 
proved that, for cardinalities $m$ for which the $k$--binomial decomposition has length $\ell (\a)<k$, 
these initial segments $I_{n,k}(m)$ are in fact the unique extremal families.  They also exhibit  nontrivial examples which show that this may not be the case when $\ell(\a)=k$. 
We note that, for fixed $k$, the set of integers with $k$--binomial decomposition of length $k$ has upper asymptotic density one
\footnote{In the interval $[\binom{n}{k}, \binom{n+1}{k}]$ there are $\binom{n-1}{k-1}$ numbers whose 
$k$--binomial decomposition has length $k$, one for each choice of an increasing sequence of length $k-1$  of numbers in  $[n-1]$, and 
$\lim_{n\to\infty} \binom{n}{k-1}/\binom{n-1}{k-1}=1$.}.
Thus, the unicity of the extremal families can only be ensured  on a thin set of cardinalities.  
This prompted the authors of \cite{furgri86} to ask about a general characterization of  the  extremal families.   The aim of this paper is to address this question, an answer to which can be found  in Theorem~\ref{thm:charac} below. 

The proof of Theorem~\ref{thm:charac} involves three ingredients. The main one  is the following new numerical inequality of binomial sequences.

\begin{lemma}\label{lem:abc} Let $m$ and $k$ be positive integers. Let    $\binom{\a}{k}\stackrel{b}{=} m$, and let $\b$ and $\c$ be
 strictly decreasing sequences of non-negative integers ($\b$ not empty, $\c$ possibly empty).
Assume that 
\begin{equation}\label{eq:abc0}
\binom{\a}{k}=\binom{\b}{k}+\binom{\c}{k-1},
\end{equation}
If
\begin{equation}\label{eq:cond0}
\b\geq_{lex} \a-1,
\end{equation}
then, for each $i\ge 0$,
\begin{align}
\binom{\a}{k-i}\le \binom{\b}{k-i}+\binom{\c}{k-i-1}.
\label{eq:abc-i}
\end{align}
Moreover, if there is equality in \eqref{eq:abc-i} for $i=1$, 
and 
$c_j\geq k-1-j-1$ and $b_j\geq k-j-1$ for all $j$ for which the terms exists, 
then the equality  \eqref{eq:abc-i} holds for all $i\ge 1$.
\end{lemma}

We note that the condition  \eqref{eq:cond0} in  Lemma \ref{lem:abc} can not be ommitted. For instance,  we have $\binom{2n-1}{n}=\binom{2n-1}{n-1}$ for every $n$ (so $\b$ is the empty sequence in \eqref{eq:abc0}) while $\binom{2n-1}{n-1}>\binom{2n-1}{n-2}$. There are sequences $\a$ and integers $k$ for which the condition is in fact necessary for the conclusion of the Lemma to hold. The non-negativity condition on $\b$ and $\c$ is also necessary (see the case $k=3$, $\a=(4,2)$, $\b=(3,2,-10)$, $\c=(-10,-42)$); the non-emptyness condition on $\b$ is necessary so that \eqref{eq:cond0} imposes a non-trivial condition (see $k=1$, $\a=(1)$, $\b=\emptyset$, $\c=(0)$); the lower bounds on $c_j$ and $b_j$ in the second part is also necessary (see the case $k=3$, $\a=(3,2,1)$, $\b=(3,0)$, $\c=(2,1)$). Moraly, all three $\a$, $\b$, and $\c$ ought to be binomial decompositions, yet the condition imposed by \eqref{eq:cond0} enlarges the context of the statement beyond the binomial decompositions for the sequences $\b$ and $\c$.

Based on the numerical inequality \eqref{eq:abc-i} with $i=1$ one can easily derive the Kruskal--Katona theorem. Unlike some existing proofs, the one derived from Lemma \ref{lem:abc} does not involve compression or any transformation of the families of $k$--subsets, and it is therefore suitable to analyze the structure of the extremal ones. Our proof is closer in spirit to the one given by Kruskal in \cite{krus63} where there is some implicit use of  \eqref{eq:abc-i}  (see \cite[Lemma~7, Section 6]{krus63}), which is however in itself not sufficient to provide the characterization of extremal families given in this paper. Furthermore,  inequality \eqref{eq:abc-i} provides also a simple proof of  a result by F\"uredi and Griggs \cite[Theorem~2.1]{furedi_families_1986} which shows that the lower shadow of an extremal family is itself extremal.  These simple proofs of both results are given in Section~\ref{sec:kk}. A common feature of these proofs is that they rely on numerical inequalities that hold independently of their potential connection to families of $k$--subsets of a ground set. 


Lemma \ref{lem:abc}  provides the following recursive characterization of extremal families, which is the main result of this paper and gives an answer to the question posed by F\"uredi and Griggs in \cite{furedi_families_1986}. We use the following notation. For a family $S\subset \binom{[n]}{k}$ of $k$--subsets of $[n]$ and $x\in [n]$ we denote by
$$
S(x)=\{X\in S: x\in X\}
$$
the family of sets in $S$ containing $x$, and by 
$$
d_S(x)=|S(x)|
$$ 
the degree of $x$ in $S$, which is always assumed to be nonzero, that is, we  assume (without loss of generality) that $\cup_{X\in S} X=[n]$. We also denote by 
$$
S(x)\setminus x=\{X\setminus x: X\in S(x)\}.
$$

\begin{theorem}\label{thm:charac} Let $n\ge k\ge 2$ and let  $S\subset \binom{[n]}{k}$ be a non-empty family of $k$--subsets of $[n]$, where $[n]$ is the support of $S$. Let
$\binom{{\mathbf a}}{k}\stackrel{b}{=}|S|$
be the $k$--binomial decomposition of $m=|S|$. For each $x\in [n]$, let
$$
\binom{{\mathbf c}(x)}{k-1}\stackrel{b}{=}|S(x)|=|S(x)\setminus x|\qquad \text{and}\qquad \binom{{\mathbf b}(x)}{k}\stackrel{b}{=}|S|-|S(x)|.
$$
(the binomial decomposition of the $0$ being the empty sequence.)

The family $S$ is {extremal} 
	if and only if,
for every $x\in [n]$, the inequality
\begin{equation}\label{eq:b>a-1}
|S|-|S(x)|\ge \binom{{\mathbf a}-1}{k}
\end{equation}
holds, and either
\begin{enumerate}[label=(\roman*)]
\item 
\label{en:t1} 
  the inequality \eqref{eq:b>a-1} is strict and 
\begin{enumerate}[label=(\roman{enumi}.\arabic*)]
\item 
\label{en:t11} 
 (inclusion) $S(x)\setminus x \subseteq \Delta (S\setminus S(x))$,
\item 
\label{en:t12} 
(extremality) both $S\setminus S(x)$ and $S(x)\setminus x$ are extremal, and
\item 
\label{en:t13} 
(numerical equality) $\binom{{\mathbf a}}{k-1}=\binom{{\mathbf b}(x)}{ k-1}+\binom{{\mathbf c}(x)}{k-2}$;
\end{enumerate}
\end{enumerate}
or
\begin{enumerate}[resume,label=(\roman*)]
\item 
\label{en:t2} %
there is equality in \eqref{eq:b>a-1} 
and 
\begin{enumerate}[label=(\roman{enumi}.\arabic*)]
\item 
\label{en:t21}  %
 (inclusion) $\Delta (S\setminus S(x))\subseteq S(x)\setminus x$, and
\item 
\label{en:t22} (extremality)  $S(x)\setminus x$ is extremal.
\end{enumerate}
\end{enumerate}
Additionally, if \ref{en:t1} or \ref{en:t2} holds for a $x\in [n]$, then $S$ is extremal.
\end{theorem}

With the characterization of Theorem \ref{thm:charac} one can show the unicity of the colex order for cardinalities $m$  whose $k$--binomial sequence $\a$ has length $\ell (\a)<k$, a result proved in \cite[Theorem~2.6]{furgri86} and \cite[Theorem~7]{mors85}. We also characterize those cardinalities where the initial segment of the  colex order is the unique extremal family.

\begin{theorem}\label{thm:ucolex} Let $n\ge k> 0$ and $0<m\le \binom{n}{k}$. Let $\binom{\a}{k}\stackrel{b}{=} m$ be the $k$--binomial decomposition of $m$. 
		
The initial segment of length $m$ in the colex order in $\binom{[n]}{ k}$  is the unique  extremal set (up to automorphism) 
if and only if either
\begin{enumerate}[label=(\roman*)]
\item 
\label{en:ucol_1} $\ell (\a)<k$, or
\item 
\label{en:ucol_2} $m=\binom{n'}{k}-1=\binom{n'-1}{k}+\binom{n'-2}{ k-1}+\cdots +\binom{n'-k}{1}$ for some $n'$ with $k<n'\le n$.
\end{enumerate}
\end{theorem}

The paper is organized as follows. The proof of Lemma  \ref{lem:abc} is postponed to Section \ref{sec:ineq}, where we state more general versions of the result as Lemma~\ref{lem:abck} and Corollary~\ref{cor.1}, and discuss some technical details related to it.  Section \ref{sec:kk} derives the Kruskal--Katona theorem and \cite[Theorem~2.1]{furedi_families_1986} from Lemma \ref{lem:abc}. The simple proof of these known results illustrates the use of Lemma \ref{lem:abc} and it is also used in the proof of  Theorem \ref{thm:charac} on the characterization of extremal families, which is the contents of   Section \ref{sec:charac}. In subsection \ref{subsec:examples} we include examples of nonextremal sets showing that the conditions in Theorem \ref{thm:charac} are independent of each other, so none of them can be ommitted in the characterization. Section \ref{sec:ucolex} contains the proof of Theorem \ref{thm:ucolex} on the unicity of the initial segments of the colex order as extremal sets. 
The paper concludes with  final remarks in Section \ref{sec:concl}. 


\section{The theorems of Kruskal-Katona and F\"uredi-Griggs}\label{sec:kk}

We next show how Lemma \ref{lem:abc} implies both Kruskal-Katona's theorem \cite{zbMATH03489128,krisud06}  as well as \cite[Theorem~2.1]{furedi_families_1986}, which we combine as Theorem~\ref{thm:kk+fg} below. 

\begin{theorem}[\cite{kat09,krus63,furgri86}]\label{thm:kk+fg}
	Let $n,k$ be positive integers with $n\geq k$ and $S\subset \binom{[n]}{k}$ be  a family of $k$--subsets of $[n]$ with cardinality $|S|=m>0$. If $\binom{\a}{k}\stackrel{b}{=} m$ then 
	\[|\Delta(S)|\geq \binom{\ba}{k-1}\]
	Furthermore, if  $|\Delta(S)|=\binom{\ba}{k-1}$, then $|\Delta^i(S)|=\binom{\ba}{k-i}$
		for each $i\geq 0$.
\end{theorem}

We will  use the following Lemma to ensure that condition \eqref{eq:cond0} in Lemma \ref{lem:abc} holds. 

\begin{lemma}\label{lem:a-1<b} Let $n,k$ be positive integers with $n> k>1$ and $S\subset \binom{[n]}{k}$ be  a family of $k$--subsets of $[n]$ with cardinality $|S|=m>1$ and support $[n]$. Let $x\in [n]$ be an element with minimum degree in $S$. Let $\binom{\a}{ k}\stackrel{b}{=} m$ be the $k$--binomial decomposition of $m$ and $\binom{\b}{k}\stackrel{b}{=} |S|-d_S(x)$. Then, $\b$ is not the empty sequence and
\begin{equation}\label{eq:a-1<b}
\b\ge_{lex}\a-1.
\end{equation}
 \end{lemma}

\begin{proof} Since $n>k>1$, $|S|>1$ and $x$ has  minimum degree, then not all the sets in $S$  contain $x$ and hence $|S|-d_S(x)>0$. 
This implies that $\b$ is not the empty sequence.
	
 Let $\binom{\a}{k}=\sum_{i=0}^t\binom{a_i}{k-i}$. If $|S |\le k$ then $\binom{\a-1}{k}=0$ and the statement trivially follows as $\binom{\b}{k}>0$.	
	 Assume $|S|>k$. The statement also follows trivially  if $|S|=\binom{n}{k}$. Assume $|S|<\binom{n}{k}$. In particular, $a_0\le n-1$ and therefore $a_i\le n-i-1$ for each $i=0,\ldots ,t$.   The average degree in $S$ of elements in $[n]$  is $\frac{k}{n}|S|$. Therefore, if $t'\le t$ is the largest subscript $i$ such that $a_i>k-i$ (which exists as $|S|>k$), we have
	\begin{equation}\label{eq:dom1}
	|S|-d_S(x)\ge |S|\left(1-\frac{k}{n}\right)=\sum_{i=0}^t\binom{a_i}{k-i}\left(1-\frac{k}{n}\right)=\sum_{i=0}^{t'}\binom{a_i}{ k-i}\left(1-\frac{k}{n}\right)+(t-t')\left(1-\frac{k}{n}\right),
	\end{equation}
	Let us show that, for each  $j\le t'$,
	\begin{equation}\label{eq:dom2}
	\sum_{i=0}^{j}\binom{a_i}{k-i}\left(1-\frac{k}{n}\right)\ge \sum_{i=0}^{j}\binom{a_i-1}{k-i}+\frac{j+1}{n}\binom{a_{j}-1}{k-j-1}.
	\end{equation}
	For $j=0$ we have, using $a_0\le n-1$,
	$$
	\binom{a_0}{k}(1-\frac{k}{n})=\binom{a_0}{ k}-\frac{k}{n}\frac{a_0}{k}\binom{a_0-1}{k-1}\ge \binom{a_0}{ k}-(1-\frac{1}{n})\binom{a_0-1}{k-1}=\binom{a_0-1}{ k}+\frac{1}{n}\binom{a_0-1}{k-1},
	$$
	and \eqref{eq:dom2} holds for $j=0$. Let $j>0$. By using induction and  $a_j\le \min\{a_{j-1}-1,n-j-1\}$, we similarly have
	\begin{align}
\sum_{i=0}^{j}&\binom{a_i}{k-i}(1-\frac{k}{n})\ge 
\sum_{i=0}^{j-1}\binom{a_i-1}{k-i}+\frac{j}{n}\binom{a_{j-1}-1}{ k-j}+(1-\frac{k}{n})\binom{a_{j}}{k-j}
\nonumber \\ \stackrel{a_{j-1}-1\geq a_j}{\ge} &\sum_{i=0}^{j-1}\binom{a_i-1}{ k-i}+\binom{a_j}{k-j}(1-\frac{k-j}{n})=\sum_{i=0}^{j}\binom{a_i-1}{ k-i}+\binom{a_j-1}{k-j-1}(1-\frac{k-j}{n}-\frac{a_j-(k-j)}{n}) \nonumber \\
\stackrel{-a_j\geq-n+j+1}{\ge} &\sum_{i=0}^{j}\binom{a_i-1}{ k-i}+\frac{j+1}{n}\binom{a_j-1}{k-j-1}. \nonumber
\end{align}
	The statement of the Lemma follows from \eqref{eq:dom1} and \eqref{eq:dom2} with  $j=t'$, as $\binom{\a-1}{ k}=\sum_{i=0}^{t'}\binom{a_i-1}{k-i}$ (the additional unit needed when $a_i=k-i$ for some $i$, comes from the integer part of the last summand).
\end{proof}

\begin{proof} (of Theorem \ref{thm:kk+fg})
We prove the result by a triple induction on $m$, $n$ and $k$. The result holds with equality when $k=1$ for all $n$ and $m$ as $\Delta(S)=\{\emptyset\}$ in this case. Moreover, when $n=k$ the result also holds since $|S|=1=\binom{k}{k}$ and $|\Delta(S)|=k=\binom{k}{k-1}$.
The result also holds for $m=1$ and all $n\geq k$ by the same reason.

Let $n>k>1$. We may assume that $n$ is an element with minimum degree in $S$, and assume that $|S|>1$, and thus $d_S(n)<|S|$. Let  $\binom{\b}{ k}\stackrel{b}{=}|S|-d_S(n)$. Partition $S=S_0\sqcup S_1$ with $S_0=S(n)=\{x\in S\:|\: n\in x\}$ and $S_1=S\setminus S_0$. 
Denote by
$$
S'_0=S_0\setminus n=\{x\setminus n\:|\: x\in S_0,n\in x\}\subseteq \binom{[n-1]}{k-1}.
$$
We observe that
\begin{equation}\label{eq:rel_both}
|\Delta(S)|\geq |\Delta(S_1)|+|\Delta(S_0')|.
\end{equation}
as all the sets in $\Delta(S_0')$ can be though of as the sets in $\Delta(S)$ containing $n$.
Let $\binom{\bc}{k-1}\stackrel{b}{=}|S_0|=|S'_0|$, so that 
\begin{equation}\label{eq:sum_is_ok}
\binom{\a}{k}=\binom{\b}{k}+\binom{\c}{k-1}.
\end{equation} By induction, 
\begin{equation}\label{eq.c_k-2}
|\Delta(S_0')|\geq \binom{\bc}{k-2}.
\end{equation}
On the other hand, we have $S_1\subseteq \binom{n-1}{k}$ and  $\binom{\b}{k}\stackrel{b}{=}|S_1|=|S|-d_S(n)$. By induction,
\begin{equation}\label{eq.b_k-1}
|\Delta(S_1)|\geq \binom{\bb}{k-1}.
\end{equation}
By Lemma  \ref{lem:a-1<b}, the conditions of Lemma~\ref{lem:abc} are fullfilled and therefore \eqref{eq:sum_is_ok} \eqref{eq:rel_both}, \eqref{eq.c_k-2} and \eqref{eq.b_k-1} give
\begin{equation}\label{eq:kk1}
|\Delta(S)|\ge \binom{\b}{k-1}+\binom{\c}{k-2}\ge \binom{\a}{k-1},
\end{equation}
completing the induction step and the proof of the first part of the Theorem.  

Suppose now that equality holds along \eqref{eq:kk1}. Thus, we have the equalities
\begin{align}
|\Delta(S_0')|&=\binom{\c}{k-2}, \label{eq:tight0}\\
 |\Delta(S_1)|&=\binom{\b}{ k-1},  \label{eq:tight1}\\
\binom{\a}{k-1}&=\binom{\b}{k-1}+\binom{\c}{ k-2},  \;\text{and} \label{eq:tight2}\\
|\Delta(S)|&= |\Delta(S_1)|+|\Delta(S_0')|. \label{eq:tight30}
\end{align}
By \eqref{eq:tight2} and the second part of Lemma~\ref{lem:abc}, for each $i\ge 0$, we have
\begin{equation}\label{eq:tight3}
\binom{\a}{k-i}=\binom{\b}{ k-i}+\binom{\c}{k-i-1}.
\end{equation}
On the other hand, by induction, equalities in  \eqref{eq:tight0} and \eqref{eq:tight1} lead to
\begin{equation}
\label{eq:tight60}
|\Delta^i(S_0')|=\binom{\c}{ k-i-1}\;\text{and} \; |\Delta^i(S_1)|=\binom{\b}{k-i},
\end{equation}
and moreover,
\begin{equation}\label{eq:tight50}
|\Delta^i(S)|= |\Delta^i(S'_0)|+|\Delta^i(S_1)|,
\end{equation}
as the equality in \eqref{eq:tight30} implies that no other sets in $\Delta(S)$ exist beyond those containing $n$ (accounted in $\Delta(S_1)$), and those not containing $n$ (accounted in  $\Delta(S'_0)$). In particular, all the sets in $\Delta^i(S)$ are either accounted in  $\Delta^i(S'_0)$ if they do not contain $n$, or accounted in $\Delta^i(S_1)$ if they contain $n$.
 
  By combining \eqref{eq:tight3} with \eqref{eq:tight60} and \eqref{eq:tight50}   we get
$$
 |\Delta^i(S)|=\binom{\c}{ k-i-1}+\binom{\b}{ k-i}= \binom{\a}{k-i},
$$
completing the proof of the second part of the statement.
\end{proof}


\section{Characterization of  extremal families}\label{sec:charac}

In this section we prove Theorem \ref{thm:charac}.

\begin{proof}[Proof of Theorem \ref{thm:charac}] Assume first that $S$ is extremal. We follow an argument by Frankl~\cite{frankl1984new}. Arguing for a contradiction, suppose that there is $i\in [n]$ such that $|S(i)|>\binom{\a-1}{k-1}$.  Let $\binom{\c}{ k-1}\stackrel{b}{=} |S(i)|>\binom{\a-1}{k-1}$ and $S_i'=S(i)\setminus i$. We have
\begin{equation}\label{eq:tight4}
|\Delta (S(i))|= |\Delta (S'_i)|+|S'_i|.
\end{equation}
as the set in $\Delta (S(i))$ either contains $i$, and hence is accounted in $\Delta (S'_i)$, or it does not contain $i$, and hence there are $|S'_i|$ of them.
By Theorem~\ref{thm:kk+fg} and our assumption on $|S(i)|$, 
\begin{equation}\label{eq:tight70}
|\Delta (S)|\ge |\Delta (S(i))|= |\Delta (S'_i)|+|S'_i|> \binom{\bc}{k-2} +\binom{\a-1}{ k-1}\geq\binom{\a-1 }{k-2}+\binom{\a-1}{ k-1}=\binom{\a}{ k-1}, 
\end{equation}
contradicting the extremality of $S$. Therefore, we reach a contradiction and conclude that, $|S(i)|\leq \binom{\a-1}{ k-1}$ for each $i\in[n]$.

The same argument shows that, if $|S(i)|=\binom{\a-1}{ k-1}$, then equality must hold along the above inequalities and  $|\Delta (S(i))|= \binom{\bc}{k-2}$ so that $S(i)$ is extremal, proving \ref{en:t22}. Moreover, there is equality in \eqref{eq:tight70} and thus $|\Delta(S)|=|\Delta(S(i))|$, proving 
\ref{en:t21}.

Suppose that $|S(i)|<\binom{\a-1}{ k-1}$. Let $\binom{\b}{ k}\stackrel{b}{=}|S\setminus S(i)|$ which is strictly positive  and  satisfies $\binom{\b}{ k}> \binom{\a-1}{k}$ by $|S(i)|<\binom{\a-1}{k-1}$. As in \eqref{eq:rel_both} we have
\begin{equation}\label{eq:tight5}
|\Delta (S)|\ge |\Delta (S\setminus S(i)|+|\Delta (S'_i)|,
\end{equation}
as adding $i$ to the subsets in $\Delta (S'_i)$ we obtain subsets in $\Delta (S)\setminus \Delta (S\setminus S(i))$. As $\binom{\b}{ k}> \binom{\a-1}{ k}$ then $\b\geq_{\text{lex}} \a-1$ (actually $\b>_{\text{lex}} \a-1$ as $\b$ is the sequence of a non-zero $k$-binomial decomposition), thus we can apply Lemma~\ref{lem:abc} and we have
$$
|\Delta (S)|\ge |\Delta (S\setminus S(i))|+|\Delta (S'_i)|\ge  \binom{\b}{ k-1}+\binom{\c}{ k-2}\ge \binom{\a}{ k-1}=|\Delta (S)|,
$$
so equality holds along the above inequalities. In particular, both $S(i)$ and $S\setminus S(i)$ are extremal, proving \ref{en:t12}, and the equality  \ref{en:t13} holds, and  there is equality in \eqref{eq:tight5} thus proving \ref{en:t11}.

The reciprocal implication follows by just considering one element
$i\in [n]$, thus showing the second part of the result.
 If $i$ is such that $S(i)<\binom{\ba-1}{k-1}$, then we can repeat the argument above to conclude that there is equality everywhere (due to the four assumptions), and hence $S$ is extremal.
If $i$ is such that $S(i)=\binom{\ba-1}{k-1}$, then the extremality on $S'_i$ implies that
$|S'_i|+|\Delta(S'_i)|=\binom{\ba}{k-1}$. Then the containment $\Delta(S\setminus S(i))\subseteq S'_i$ implies that $\Delta(S)$ (whose sets are either in $\Delta(S\setminus S(i))$, or in $\Delta (S'_i)\times \{i\}$ or in $S_i'$), have all the sets accounted for with just $|S'_i|+|\Delta(S'_i)|$, so $|S'_i|+|\Delta(S'_i)|\geq |\Delta(S)|$ implies that there is equality in \eqref{eq:tight70} and the extremality of $S$ follows.
\end{proof}


\subsection{Examples on the necessity of the conditions}\label{subsec:examples}

In this section  we give  examples showing that the conditions for the characterization of extremal sets in Theorem \ref{thm:charac} are  mutually independent, so none of them  can be ommitted.  Conditions \ref{en:t11} and \ref{en:t12} are of structural nature while condition \ref{en:t13} is of quantitative nature. A natural question is to ask if the structural conditions may imply the quantitative one. We next give an example of a set $S$ which satisfies conditions \ref{en:t11} and \ref{en:t12} for each $x\in [n]$ but fails to be extremal because the condition \ref{en:t13} fails to hold, showing that the last quantitative  condition is independent from the structural ones.

\begin{example} For each $m\geq 3$, $n\ge m+3\ge 6$ we denote by $S_{n,k,m}$ the family
	$$
	S_{n,k,m}=\binom{[n]}{ k}\setminus \left[\cup_{\{i,j\}\in \binom{[m]}{2}}\nabla^{k-2}(\{i,j\}) \right],
	$$
	where $\nabla^{k-2}(X)$ denotes the $(k-2)$--iterated upper shadow of the family $X$. The family $S_{n,k,m}$ consists of all $k$--subsets of $[n]$ which  contain none of  the sets  $\{i,j\}$ with $i\neq j$ and $i,j\leq m$. Its cardinality is
	$$
	|S_{n,k,m}|=\binom{n}{k}-\sum_{i=2}^{k}\binom{m}{i}\binom{n-m}{k-i}=\binom{n-m}{ k}+\binom{n-m}{ k-1}\binom{m}{ 1},
	$$
	and its lower shadow is the family of all $(k-1)$--subsets of $[n]$ containing no pair of elements in $[m]$:
	$$
	\Delta(S_{n,k,m})=S_{n,k-1,m}.
	$$
	For every $x\in [n]$ we have
	{\small \begin{align*}
	|S_{n,k,m}(x)|&=\begin{cases}|S_{n-1,k-1,m}|, &x\in [n]\setminus [m],\\
	|S_{n-1,k-1,m-1}|,& x\in [m],\end{cases} &  
	|\Delta(S_{n,k,m}(x))|&=\begin{cases}|S_{n-1,k-2,m}|, &x\in [n]\setminus [m],\\
	|S_{n-1,k-2,m-1}|,& x\in [m],\end{cases},\\
	|S_{n,k,m}\setminus S_{n,k,m}(x)|&=\begin{cases}|S_{n-1,k,m}|, &x\in [n]\setminus [m],\\
	|S_{n-1,k,m-1}|,& x\in [m],\end{cases},&	
	|\Delta (S_{n,k,m}\setminus S_{n,k,m}(x))|&=\begin{cases}|S_{n-1,k-1,m}|, &x\in [n]\setminus [m],\\
	|S_{n-1,k-1,m-1}|,& x\in [m],\end{cases}.
	\end{align*}}

	The example will be constructed from $S_{n,k,m}$ by deleting a family $L\subset \binom{[n]\setminus [m]}{ k}$. For every such choice of $L$, let 
	$$
	S=S_{n,k,m}\setminus L.
	$$
	Since  every set $y$ in the shadow of   $\binom{[n]\setminus [m]}{ k}$ is contained in the shadow of  $y\cup \{i\}\in S$ for each $i\in [m]$ we have
	$$
	\Delta (S)=\Delta (S_{n,k,m}).
	$$
	Moreover, by the same reason, for every $x\in [n]$, we also have 
	$$
	\Delta (S(x))=\Delta (S_{n,k,m}(x))\quad \text{and}\quad \Delta (S\setminus S(x))=\Delta (S_{n,k,m}\setminus S_{n,k,m}(x)).
	$$ 
	For  positive integers $t$ and $r$ with $t\geq r$, let $n=tk+m$ and let $L$ be an $r$--regular family of $k$--subsets of $[n]\setminus [m]$, each element in $[n]\setminus [m]$ belongs to $r$ subsets of $L$ (such a family exists since $n-m=tk$ is a multiple of $k$, and $|L|=tr$). 
	
Appropriate choices of the parameters $k,m$ and $t$ provide examples of  sets $S$  satisfying conditions \ref{en:t11} and \ref{en:t12} for each $x\in [n]$ but fail to be extremal because the condition \ref{en:t13} does not hold.

For example, for $m=k=4$, $t=29$ and $r=2$ we have $n=tk+4=120$ and  $|L|=tr=58$ and $S_{n,k,m}$ satisfies
\begin{align*}
|S_{n,k,m}|&=\binom{119}{ 4}+\binom{112}{ 3}+\binom{104}{ 2}+\binom{ 58}{ 1}\\
|\Delta (S_{n,k,m})|=|S_{n,k-1,m}|&=\binom{119}{ 3}+\binom{112}{ 2}+\binom{105}{ 1},
\end{align*}
so that $S_{n,k,m}$ is an extremal set. However, for the set $S=S_{n,m,k}\setminus L$
\begin{align*}
|S|=|S_{n,k,m}|-|L|&= 	
\binom{119}{ 4}+\binom{112}{ 3}+\binom{104}{ 2},
\end{align*}	
and $\Delta (S)=\Delta (S_{n,k,m})=S_{n,k-1,m}$, so that  $S$ is not extremal. Nevertheless, the families $S(x)$ and $S\setminus S(x)$ are extremal for all $x\in [n]$. Indeed, if $x\in [n]\setminus [m]$, then $x$ is contained in $r=2$ sets in $L$ and  
\begin{align*}
|S(x)|=|S_{n-1,k-1,m}(x)|-2&=	
\binom{118}{ 3}+\binom{111}{ 2}+\binom{102}{ 1}\\ 
|\Delta (S(x))|=|S_{n-1,k-2,m}|&=\binom{118 2}+\binom{112}{ 1},\\ 
|S\setminus S(x)|=|S_{n-1,k,m}|-(tr-2)&=\binom{118}{ 4}+\binom{111}{ 3}+\binom{103}{ 2}+\binom{1}{ 1},\\
|\Delta (S\setminus S(x)|=|S_{n-1,k-1,m}|&=\binom{118}{ 3}+\binom{111}{ 2}+\binom{104}{ 1}.
\end{align*}
On the other hand, if $x\in [m]$ then $S(x)=S_{n,k,m}(x)$, which is extremal since $S_{n,k,m}$ is extremal itself, and 
\begin{align*}
|S\setminus S(x)|=|S_{n-1,k,m-1}|-tr&=\binom{118}{ 4}+\binom{114}{ 3}+\binom{112}{ 2}+\binom{52}{ 1},\\
|\Delta (S\setminus S(x)|=|S_{n-1,k-1,m}|&=\binom{118}{ 3}+\binom{114}{ 2}+\binom{113}{ 1}.
\end{align*}
Moreover, in all cases we have $S(x)\setminus \{x\}\subset \Delta (S\setminus S(x))$ so that the set $S$ fails to be extremal because \ref{en:t13} alone does not hold. \qed
\end{example}

The structural conditions are also easily seen to be independent of each other. Some examples are given below. 
We recall that $I_{n,k}(m)$ denotes the initial segment of length $m$ in the colex order. We use the notation $S\vee S'=\{x\cup y: x\in S, y\in S'\}$ to denote the family of all sets obtained by unions of a set  in $S$ and a set in $S'$.

\begin{example} A set $S$ satisfying \eqref{eq:b>a-1} strictly and \ref{en:t11} and \ref{en:t13} but not \ref{en:t12}.  Let $k\ge 3$ and $n>k$. Let  $m_1= \binom{n-1}{ k}+\binom{n-2}{ k-1}\stackrel{b}{=}\binom{\b }{ k}$
 and  $m_2=\binom{n-2}{ k-2}\stackrel{b}{=}\binom{\c}{ k-1}$. Let $m'_1=\binom{n-1}{ k}$ and let $B$ be the family of  $(k-1)$--subsets of $[n+1,2n-2]$.  Define
$$
S=(I_{n,k}(m'_1)\cup (B\vee \{2n\}))\cup (I_{n,k-1}(m_2)\vee \{n\}).
$$
We have,
$$
|S|=\binom{n-1}{ k}+\binom{n-2}{ k-1}+\binom{n-2}{ k-2}=\binom{n}{ k}\stackrel{b}{=}\binom{\a}{ k},
$$
so that $\b>_{\text{lex}}\a-1$ and  \eqref{eq:b>a-1} holds. Now 
$$
\binom{\a}{ k-1}=\binom{n}{ k-1}=\binom{n-1}{ k-1}+\binom{n-2}{ k-2}+\binom{n-2}{ k-3}=\binom{\b}{ k-1}+\binom{\c}{ k-2},
$$
and \ref{en:t13} holds. Moreover, with the choice $x=n$, $S(x)\setminus x$ is an initial segment of  length $\binom{n-2}{ k-2}$ which is contained in the initial segment of length $\binom{n-1}{ k-1}$ contained  in $\Delta (S\setminus S(x))$, so \ref{en:t11} holds.  

However, $S\setminus S(x)$ is not extremal, as
$$
|\Delta (S\setminus S(x))|=\binom{n-1}{ k-1}+|\Delta (B)|=\binom{n-1}{ k-1}+\binom{n-2}{ k-2}+\binom{n-2}{ k-3}>\binom{n-1}{ k-1}+\binom{n-2}{ k-2}.
$$
and \ref{en:t12} does not hold.

By choosing instead 
\[
S=I_{n,k}(m_1)\cup (C\vee \{n+1\}),
\]
with $C$ an arbitrary nonextremal family of $(k-1)$-sets with cardinality $|C|=\binom{n-2}{ k-2}$ satisfying $C\subset \Delta (I_{n,k}(m_1))$
 we obtain a set satisfying \eqref{eq:b>a-1} and, with the choice $x=n+1$, also satisfying \ref{en:t11} and \ref{en:t12} but now $S(x)\setminus x$ is not extremal. \qed
\end{example}

\begin{example} A set $S$ satisfying \eqref{eq:b>a-1} strictly and \ref{en:t12} and \ref{en:t13} but not \ref{en:t11}.  Let $k\ge 3$ and $n>k$. Let  $m_1= \binom{n-1}{ k}+\binom{n-2}{ k-1}\stackrel{b}{=}\binom{\b }{ k}$
	 and  $m_2=\binom{n-2}{ k-2}\stackrel{b}{=}\binom{\c}{ k-1}$.  Define
$$
S=I_{n,k}(m_1)\cup (I_{[n+2,2n],k-1}(m_2)\vee \{2n+1\}),
$$
where $I_{[n+2,2n],k-1}(m_2)$ is an initial segment of length $m_2$ of $(k-1)$--subsets of the interval $[n+2,2n]$.  As before, we have,
$$
|S|=\binom{n-1}{ k}+\binom{n-2}{ k-1}+\binom{n-2}{ k-2}=\binom{n}{ k}\stackrel{b}{=}\binom{\a}{ k},
$$
so that $\b>_{\text{lex}}\a-1$
 and  \eqref{eq:b>a-1} holds. We also have 
$$
\binom{\a}{ k-1}=\binom{n}{ k-1}=\binom{n-1}{ k-1}+\binom{n-2}{ k-2}+\binom{n-2}{ k-3}=\binom{\b}{ k-1}+\binom{\c}{ k-2},
$$
and \ref{en:t13} holds. Moreover, with the choice $x=2n+1$, both $S(x)$ and $S\setminus S(x)$ are initial segments, and hence extremal sets, so that \ref{en:t12} holds, but $S(x)\setminus x$ is not contained in $\Delta (S\setminus S(x))$, so \ref{en:t11} does not hold.\qed 
\end{example}

\section{The key inequality}\label{sec:ineq}

In this section we give the proof of Lemma \ref{lem:abc}. As mentioned in the Introduction, Lemma \ref{lem:abc} follows from the slightly more general Lemma~\ref{lem:abck} and Corollary~\ref{cor.1} below. The first statement shows the case $i=1$ of \eqref{eq:abc-i} in a more general form (the reason being mostly technical, as the base cases of the induction argument we have fit this scenario better). The second statement, Corollary~\ref{cor.1}, shows the inequaliy \eqref{eq:abc-i} for $i>1$ and deal with the case of equality in Lemma \ref{lem:abc}, plus some further information that allows us to show the characterization of the cardinalities for which the colex is the unique extremal family (up to isomorphism) in Theorem~\ref{thm:ucolex}.

 

\begin{lemma}\label{lem:abck} Let $m$ and $k,k',k''$ be positive integers. Let   $\a,\b,\c$ be such that $\binom{\a}{ k}\stackrel{b}{=} m$, and $\b$ and $\c$ be
	\begin{itemize}
		\item  strictly decreasing sequences of non-negative integers
		\item $b_j\geq k'-j-1$ and $c_j\geq k''-j-1$ for all $j$ for which they exists ($\b$ or $\c$ may be the empty sequence).
	\end{itemize} 
Assume that 
\begin{equation}\label{eq:abc01}
\binom{\a}{ k}=\binom{\b}{ k'}+\binom{\c}{ k''},
\end{equation}
If  $k',k''\ge k$, then
\begin{align}
\binom{\a}{ k-1}\le \binom{\b}{ k'-1}+\binom{\c}{ k''-1}.\label{eq:abc1}
\end{align}
and
\begin{align}
\binom{\a-1}{ k-1}\le \binom{\b-1}{ k'-1}+\binom{\c-1}{ k''-1}.\label{eq:abc11}
\end{align}
Moreover, if $k'=k\geq 2$ and $k''=k-1$ then \eqref{eq:abc1} and \eqref{eq:abc11} holds provided that $\b$ is not empty and
\begin{equation}\label{eq:cond1}
\b\ge_{\text{lex}}  \a-1.
\end{equation}
\end{lemma}


\begin{corollary}\label{cor.1} Let $\a$ be a $k$-binomial decomposition of a positive integer, and $\b$ and $\c$ be possibly empty strictly decreasing sequences of non-negative integers, with $b_j\geq k-j-1$ and $c_j\geq k-1-j-1$ for all $j$ for which they exists and satisfying
	\begin{equation}
	\label{eq:hip1}
	\binom{\a}{ k}\le \binom{\b}{ k}+\binom{\c}{ k-1} 
	\end{equation}
	If $\b$ is not empty and $\b\ge_{lex}\a-1$ then, for each $i\ge 1$, 
	\begin{equation}\label{eq:ai}
	\binom{\a}{ k-i}\le \binom{\b}{ k-i}+\binom{\c}{ k-i-1},
	\end{equation}
	and
	\begin{equation}\label{eq:ai_m_1}
	\binom{\a-i}{ k-i}\le \binom{\b-i}{ k-i}+\binom{\c-i}{ k-i-1},
	\end{equation}
	Moreover, if there is equality in \eqref{eq:ai} for $i= 1$ and $i=0$, 
	then, 
	\begin{itemize}
		\item for each   $0\le j\le i$,
		\begin{equation}\label{eq:a-j}
		\binom{\a-j}{ k-i}= \binom{\b-j}{ k-i}+\binom{\c-j}{ k-i}.
		\end{equation}
		\item if $\textbf{a}=(a_0,\ldots,a_t)$ with $a_i>a_{i+1}$ and $t<k-1$, then
		\begin{equation}\label{eq:conseq1}
		\binom{\a}{k}\doteq \binom{\b}{k} + \binom{\c}{k-1}
		\end{equation}
		In particular, the only choices for $\c$ and $\b$ are:
		\begin{equation}\label{eq:numerical_colex_uniq}
		\begin{cases}
		\c=(a_0-1,\ldots,a_{i-1}-1,a_{i+1},\ldots,a_t) &\text{ and } \b=(a_0-1,\ldots,a_{i-1}-1,a_{i})  \\
		\c=(a_0-1,\ldots,a_{i-1}-1) &\text{ and } \b=(a_0-1,\ldots,a_{i-1}-1,a_{i},a_{i+1},\ldots,a_t) 
		\end{cases}
		\end{equation}
		for each $i\in[0,t+1]$ such that $a_{i-1}-1>a_{i}$ (with no condition when $i=t+1$ or $i=0$).
	\end{itemize}
\end{corollary}

\subsection{Proof of Lemma~\ref{lem:abc} from Lemma~\ref{lem:abck} and Corollary~\ref{cor.1}}

If $\a$ is a k-binomial decomposition of a strictly positive number, then to satisfy the inequality $\b\geq_{\text{lex}}\a-1$, it suffices to ask $\b$ to be, from among the strictly decreasing sequences of non-negative numbers, one for which $b_i\geq k-i-1$.

If $\b$ and $\c$ are strictly decreasing sequences of non-negative integers, then we let $\b'$ and $\c'$ be those elements for which $b_i\geq k-i-1$ and $c_i\geq k-i-1-1$, then $\b'$ and $\c'$ are also strictly decreasing sequences of non-negative integers, and we have
\[
\binom{\b}{k}=\binom{\b'}{k}, \binom{\c}{k-1}=\binom{\c'}{k-1}, \binom{\b}{k-1}=\binom{\b'}{k-1}, \binom{\c}{k-1-1}=\binom{\c'}{k-1-1}
\]
the last two equalities following by the non-negativity assumption (then $c_{k-1-1}\geq 0=k-(k-1-1)-1-1$ and $b_{k-1}\geq 0=k-(k-1)-1$, if they exists).
This shows why it is enough to show \eqref{eq:abc-i} when $i=1$ for those decreasing sequences satisfying $b_i\geq k-i-1$ and $c_i\geq k-i-1-1$.
In general, for $i\geq 1$, we have that
\begin{equation} \label{eq:cond_ench}
\binom{\b}{k-i}\leq \binom{\b'}{k-i}, \binom{\c}{k-i-1}\leq\binom{\c'}{k-i-1}
\end{equation}
and thus \eqref{eq:abc-i} follows from Corollary~\ref{cor.1}. Observe also that, in order to show \eqref{eq:abc-i} in the case of equality, we are also demanding the condition $b_i\geq k-i-1$ and $c_i\geq k-i-1-1$, as in Corollary~\ref{cor.1}; the reason being that, if there is a term with $0\leq b_i<k-i-1$ or $0\leq c_i<k-i-1-1$, then the there is an $i$ for which \eqref{eq:cond_ench} follows with an strict inequality, this, combined with Corollary~\ref{cor.1}, shows that, in order to demand equality in \eqref{eq:abc-i}, we should impose the condition $b_i\geq k-i-1$ and $c_i\geq k-i-1-1$ on the decreasing sequences of non-negative integers.


\subsection{General comments for Lemma~\ref{lem:abck} and Corollary~\ref{cor.1}} 

The condition $b_j\geq k'-j-1$ and $c_j\geq k''-j-1$ for all $j$ for which they exists on $\b$ and $\c$ can be removed to show \eqref{eq:abc1} and \eqref{eq:ai}, but is needed to show \eqref{eq:abc11} and \eqref{eq:a-j}.

The non-negativity  condition on the coefficients, and the fact that they are strictly decreasing, implies that, although there is no restriction on $\ell(\c)$ or $\ell(\b)$, if there exists an element with  $b_j'=k'-j-1$ (or resp. $c_j'=k''-j-1$), then $\ell(\b)\leq k'$ (or, resp. $\ell(\c)\leq k''$).

The condition $k-j-1\leq b_j$, the fact that they are strictly decreasing, and $b_j\geq 0$, also implies that if
$\binom{\b }{ k}>0$ and $\binom{\b' }{ k}\stackrel{b}{=}\binom{\b }{ k}$ is its $k$-binomial decomposition, then, for each $i\geq 0$
\[
\binom{\b'-i }{ k-i}\leq\binom {\b-i}{ k-i}
\]
with some inequality being strict only whenever there exists a $b_j$ with $k-j-1=b_j$, as then the term $\binom{b_j-(k-j)+1}{ k-j-(k-j)+1}$ turns from $0$ to something positive by increasing $i$ from $(k-j-1)$ to $(k-j)$.


%

%
%
%
%
%
%

\begin{observation} \label{obs:minus1}
	If $\b$ is a $k$-binomial decomposition, then, for every $i> 0$, if 
	\begin{itemize}
		\item $\b_i$ is defined as $\binom{\b_i}{ k-i}\stackrel{b}{=}\binom{\b}{ k-i}$ and 
		\item $\b_i'$ is defined as $\binom{\b_i'}{ k-i}\stackrel{b}{=}\binom{\b-i}{ k-i}$ 
	\end{itemize} then
	\[
	\binom{\b_i}{ k-i-1}=\binom{\b}{ k-i-1}, \qquad \binom{\b_i'-1}{ k-i-1}=\binom{\b-i-1}{ k-i-1}
	\]
\end{observation} 

\begin{proof}
	If there is no term $b_{k-i}$ in $\b$, then the result holds trivially as $\b_i=\b$ and $\b_i'=\b-i$. If there is a term $b_{k-i}$, then 
	\[
	{\b\choose k-i}=\sum_{j=1}^{k-i-1} {b_i \choose k-i-j}+{b_{k-i}\choose 0}=\sum_{j=1}^{k-i-1} {b_i \choose k-i-j}+1
	\]
	so let $s$ be the largest integer, $s\geq i+1$ for which $b_{k-s-1}-1=b_{k-s}$. Then
	\[
	(\b_i)_j=\begin{cases}
	b_j & \text{ if $s$ exists and $j< k-s-1$}\\
	b_{k-s-1}+1 &\text{ if $s$ exists}\\
	\text{no element} &\text{ if $s$ exists and $j> k-s-1$}\\
	b_j & \text{ if $s$ does not exist and $j\neq k-i-1$}\\
	b_{k-i-1}+1 & \text{ if $s$ does not exist and $j= k-i-1$}\\
	\end{cases}
	\]
	one can then check that the equality ${\b_i\choose k-i-1}={\b\choose k-i-1}$ holds (the point being that the index $s$, if it exists for one value, it will exists also for the next value of $i$, until it is exhausted, in which case, the other term $+1$ when $s$ does not exists takes over). The equality ${\b_i'-1\choose k-i-1}={\b-i-1\choose k-i-1}$ follows similarly.
\end{proof}

\begin{claim} \label{cl.3}
	In Lemma~\ref{lem:abck} and Corollary~\ref{cor.1} it suffices to show the statements whenever
	\begin{equation}\label{eq:assumption_lem}
	\text{$\b$ and $\c$ are $k'$ and $k''$-binomial decompositions, and $\b>_{\text{lex}}\a-1$ in the case $(k',k'')=(k,k-1)$.}
	\end{equation}
\end{claim}
\begin{proof}[Proof of Claim~\ref{cl.3}]
If $\b$ and $\c$ are strictly decreasing sequences of non-negative integers, then we let $\b'$ and $\c'$ be those elements for which $b_i\geq k'-i$ and $c_i\geq k'-i$, then $\b'$ and $\c'$ are also strictly decreasing sequences of non-negative and are $k'$ and $k''$ binomial decompositions respectively (perhaps the empty sequence). Then we have
\begin{align*}
\binom{\b}{k'}&=\binom{\b'}{k'}, \binom{\c}{k''}=\binom{\c'}{k''}, \text{ and thus} \nonumber \\
\binom{\a}{k}&=\binom{\b'}{k'}+\binom{\c'}{k''},
\end{align*}
and also
\begin{align*}
\binom{\b}{k'-i}&\geq \binom{\b'}{k'-i}, \binom{\b-i}{k'-i}\geq \binom{\b'-i}{k'-i}\nonumber \\
\binom{\c}{k''-i}&\geq \binom{\c'}{k''-i}, \binom{\c-i}{k''-i}\geq \binom{\c'-i}{k''-i}\nonumber 
\end{align*}
therefore, if the inequalities \eqref{eq:abc1},\eqref{eq:abc11}, \eqref{eq:ai}, \eqref{eq:ai_m_1} can be shown for $\b'$, $\c'$, then they also hold for $\b,\c$. Further, if $\b'\neq \b'$ and $\c'\neq \c$, then the inequality for $i=1$ is strict, and thus if we show the result with equality for binomial decompositions, then it does not hold for sequences not corresponding to non-binomial decompositions.

Let us now assume that $\b=_{\text{lex}} \a-1$.
Then, as $\a$ is a $k$-binomial decomposition, all the terms of $\a-1$ are $a_i-1\geq k-i-1$. Furthermore, the equalities:
${\a \choose k}= {\a-1 \choose k}+{\a-1 \choose k-1}$,
${\a \choose k}= {\b \choose k}+{\c \choose k-1}$, and $\b=_{\text{lex}} \a-1$, implies that ${\a-1 \choose k}={\b \choose k}$ and thus ${\a-1 \choose k-1}={\c \choose k-1}$. As $\a$ is a $k$ binomial decomposition, $\a-1$ is a $k-1$-binomial decomposition (perhaps with the exception that ${\a-1 \choose k-1}$ has a term of the type ${a_{k-1}-1\choose 0}$); however, it is not hard to see that the condition ${\a-1 \choose k-1}={\c \choose k-1}$ implies that
${\a-1 \choose k-1}={\c' \choose k-1}$ (as before, $\c'$ is obtained from $\c$ by removing the terms with $k-i-1-1=c_i$, then all these terms in $\c$ contribute as $0$ contribute as $0$ in the binomial expression ${\c \choose k-1}$) and that actually we can let ${\a-1 \choose k-1}\stackrel{b}{=}{\c' \choose k-1}$, which also implies that ${\a-1 \choose k-1-1}={\c' \choose k-1-1}$ (see Observation~\ref{obs:minus1}); now the fact that $\b=_{\text{lex}} \a-1$ implies that ${\a-1 \choose k-1}={\b \choose k-1}$ which, together with ${\a \choose k-1}\doteq {\a-1 \choose k-1}+{\a-1 \choose k-1-1}$ shows \eqref{eq:abc1} for the case $\b=_{\text{lex}} \a-1$.
The arguments for \eqref{eq:abc11}, \eqref{eq:ai} and \eqref{eq:ai_m_1} follows similarly.

Finally, observe that if $\b>_{\text{lex}} \a-1$ due to an coefficient of the type $\binom{i-1}{i}$ (so $\b'=_{\text{lex}} \a-1$ by removing some binomial coefficients of the type $\binom{i-1}{i}$ from $\b$ to form $\b'$), and letting $\c'$ obtained by removing those elements with $c_i<k-1-i$, then the above argument shows that $\binom{\a}{k}=\binom{\b'}{k}+\binom{\c'}{k-1}$ yet
$\binom{\a}{k-1}<\binom{\b'}{k-1}+\binom{\c'}{k-1-1}$ as $\binom{\a}{k-1}=\binom{\b}{k-1}+\binom{\c'}{k-1-1}$ and $\binom{\b}{k-1}<\binom{\b'}{k-1}$.
Also, if $\b'$ (obtained by removing those $b_i=k-i-1$ from $\b$) is such that $\b'<_{\text{lex}} \a-1$, yet $\b\geq_{\text{lex}} \a-1$, then there exists a $\b''$ obtained from $\b$ by removing some (and not all) elements with $b_i=k-i-1$ so that $\b''=_{\text{lex}} \a-1$. Then the argument discussed above shows that equality holds for $\b''=_{\text{lex}} \a-1$ and $\c'$ throughout, and strict inequality follows for $\b$ (if some element has been removed), and strict inequality in the other direction for $\b'$ in \eqref{eq:abc1} and \eqref{eq:ai}. The cases  \eqref{eq:abc11} and \eqref{eq:ai_m_1} follow with strict inequalities depending on which elements have been removed from $\b$ to form $\b''$.
 This completes the proof of the claim.
\end{proof}

\textbf{Case $\b=_{\text{lex}} \a-1$.} Then, as $\a$ is a $k$-binomial decomposition, all the terms of $\a-1$ are $a_i-1\geq k-i-1$. Furthermore, the equalities:
${\a \choose k}\doteq {\a-1 \choose k}+{\a-1 \choose k-1}$,
${\a \choose k}\doteq {\b \choose k}+{\c \choose k-1}$, and $\b=_{\text{lex}} \a-1$, implies that ${\a-1 \choose k}={\b \choose k}$ and thus ${\a-1 \choose k-1}={\c \choose k-1}$. As $\a$ is a $k$ binomial decomposition, $\a-1$ is a $k-1$-binomial decomposition (perhaps with the exception that ${\a-1 \choose k-1}$ has a term of the type ${a_{k-1}-1\choose 0}$); however, it is not hard to see that the condition ${\a-1 \choose k-1}={\c \choose k-1}$ implies that
${\a-1 \choose k-1}={\c' \choose k-1}$ (all the terms in $\c$ not in $\c'$ contribute as $0$ in the binomial expression) and that actually ${\a-1 \choose k-1}\stackrel{b}{=}{\c' \choose k-1}$, which also implies that ${\a-1 \choose k-1-1}={\c' \choose k-1-1}$ (see Observation~\ref{obs:minus1}); now the fact that $\b=_{\text{lex}} \a-1$ implies that ${\a-1 \choose k-1}={\b \choose k-1}$ which, together with ${\a \choose k-1}\doteq {\a-1 \choose k-1}+{\a-1 \choose k-1-1}$ shows \eqref{eq:abc1} for the case $\b=_{\text{lex}} \a-1$.

\textbf{Case $\b>_{\text{lex}} \a-1$.}
If the inequality $\b>_{\text{lex}} \a-1$ follows from an inequality $b_i'>_{\text{lex}} a_i-1$, then $b_j'=b_j> k-j-1$ for $j\in[0,i]$ (as $\a$ is a $k$-binomial decomposition), which means that $\b'>_{\text{lex}} \a-1$, and we may argue with the pair $(\b',\c')$ instead of $(\b,\c)$; if the inequality $\b>_{\text{lex}} \a-1$ occurs only by the existence of a $b_i$ for which there is no corresponding $a_i$ (as it does not exists), and yet $\b'<_{\text{lex}} \a-1$, then we have $b_i=a_i-1$ for those $i$'s for which $a_i$ exists; letting $\b''=_{\text{lex}}\a-1$ we have $\b''<_{\text{lex}} \b$ which implies (by the non-negativity condition on the coefficients) 
${\b''\choose k-1}\leq {\b\choose k-1}$, also that
${\a \choose k}={\b''\choose k}+{\c'\choose k-1}$ and hence if we can show the result for $(\b'',\c')$ then the result also holds for $(\b,\c)$. We can apply the argument for the preceeding paragraph to show that, $\b''=_{\text{lex}}\a-1$, then \eqref{eq:abc1} holds for the pair $(\b'',\c')$, and thus \eqref{eq:abc1} holds for this case.



We give a sketch of the proof of Lemma \ref{lem:abck} in Section~\ref{subsec:sketch} below.
We first discuss the notion of translation invariant identities of binomial sums which is used in the proof and also allows us to identify the cases of equality in \eqref{eq:abc1} in relevant cases. The proof of the remaining cases of \eqref{eq:abc-i} in Lemma \ref{lem:abc} (including the ``moreover'' part) is given in Section~\ref{sec:proof_equality} as Corollary~\ref{cor.1}.

\subsection{Translation invariance}

We say that an identity $\sum_{i}\alpha_i{x_i\choose y_i}=\sum_j \alpha_j'{x'_j\choose y'_j}$ of two  sums of binomial coefficients is {\it translation invariant} (or simply invariant) if, for every pair of integers $r,s$ we have
$$
\sum_{i}\alpha_i{x_i+r\choose y_i+s}=\sum_j \alpha_{j}'{x'_j+r\choose y'_j+s}.
$$
We use the notation $\doteq$ as in  
$
\sum_{i}\alpha_i{x_i\choose y_i}\doteq \sum_j \alpha_j'{x'_j\choose y'_j}
$
to indicate that the identity is translation invariant.

The \emph{binomial recurrence}
\begin{equation}\label{eq.bin_rec}
{n\choose k}\doteq {n-1\choose k}+{n-1\choose k-1}\end{equation} is an example of an invariant identity. By applying it term by term, the  binomial recurrence  naturally translates to binomial sequences.

\begin{lemma} For a sequence $\ba$ and integer $k$ we have
	$$
	{\ba\choose k}\doteq {\ba-1\choose k}+{\ba -1\choose k-1}.
	$$
	\end{lemma}

In particular, repeated application of the binomial recurrence to a binomial sum $\sum_{i}\alpha_i{x_i\choose y_i}$ gives rise to an invariant  identity. For example, for $i\ge 1$, the two usual binomial identities, that will be used in the proof of Lemma \ref{lem:abck}, 
\begin{align}
{n\choose k}&\doteq {n-1\choose k}+{n-2\choose k-1}+\cdots +{n-i\choose k-i+1}+{n-i\choose k-i}, \label{eq:diagonal}\\
{n\choose k}&\doteq {n-1\choose k-1}+{n-2\choose k-1}+\cdots +{n-i\choose k-1}+{n-i\choose k}\label{eq:vertical},
\end{align}
are translation invariant. 

We remark that a translation invariant identity may involve ``hidden'' binomial coefficients. For example, the identity
$$
{1\choose 0}={0\choose 0}+{0\choose -1},
$$
is translation invariant, while the identity
$$
{1\choose 0}={0\choose 0},
$$
is not.

The following Proposition shows that, in fact, all translation invariant identities arise from repeated application of the binomial recurrence \eqref{eq.bin_rec}.

\begin{proposition}[Translation invariant characterization] \label{prop:char_ti}
	If $$
	A:=\sum_{i}\alpha_i{x_i \choose y_i}\doteq\sum_j \alpha_j'{x'_j\choose y'_j}=:B
	$$
	for some real numbers $\alpha_i, \alpha_j'$, and integers $x_i$, $x_i'$, $y_i'$ and $y_i$  and both $A$ and $B$ being finite sums, then
	there exists a 
	$C=\sum_{i}\beta_i{x_i\choose y_i}$ such that
	$A\doteq C$, $B\doteq C$, and $C$ is obtained from $A$ and from $B$ by repeated application of the binomial recurrence and cancellation of identical terms.
\end{proposition}

\begin{proof}[Proof of Proposition~\ref{prop:char_ti}]
	By moving $B$ to the left hand side, the statement is equivalent to show $A\doteq 0$ if and only if we can transform $A$ into $0$ by a sequence of applications of the binomial recurrence  ${n\choose k}= {n-1\choose k-1}+{n-1 \choose k},$ and cancellation of identical terms; both these operations are translation invariant.
	For convenience, let us change the index of summation in $A$ to write
	\[
	A=\sum_{t\in \Z} \sum_{j\in \Z} \alpha_{j,t} {j \choose t}\doteq 0,
	\]
Let $j_0$ and $j_m$ be respectively  the maximal and minimal  values of $j$ for which $\alpha_{j,t}\neq 0$. By repeated use of \eqref{eq:diagonal} or \eqref{eq:vertical}, we can reduce the binomial numerators and find $A'\doteq A$ with
\[
A'=\sum_{t} \alpha_{j_m,t}' {j_m \choose t}
\]
Now we claim that $\alpha_{j_m,t}'=0$ for each $t$. Indeed, assume that is not the case and let $t_0$ and $t_m$ be respectively maximal and minimal index $t$ for which $\alpha_{j_m,t}'$ is nonzero. Then, by translation invariance,
	$$
0=\sum_{t_m\le t\le t_0} \alpha_{j_m,t} {j_m \choose t-t_0}=\alpha_{j_m,t_0} {j_m \choose 0}
$$
and since ${j_m \choose 0}=1$ then $\alpha_{j_m,t_0}=0$. Then the same reasoning applies for all indices between $t_m$ and $t_0$ by an appropriate shifting of the binomial denominators.

\end{proof}

\subsection{A geometric representation}\label{sec:geom_repr}

The following geometric representation of binomial coefficients will be handy for the developments in this section. A binomial coefficient ${y\choose x}$ is identified with a point $(x,y)$ in the two--dimensional integer lattice $\Z^2$. By a diagonal in this lattice we mean a line $y=x+a$ for some $a$ and we say that ${y\choose x}$ lies on the diagonal $y-x$. Figure \ref{fig:ex1} illustrates the geometric representation of the two invariant equalities \eqref{eq:diagonal} and \eqref{eq:vertical}. Note that each $k$-binomial decomposition gives at most $k$ points on the positive quarter plane, strictly between  the vertical line $x=0$ and the diagonal $y=x-1$, which can be seen as the \emph{boundary} of the cone of binomial decompositions. 

\begin{figure}[h]
	\begin{center}
		\begin{tikzpicture}[scale=0.5]
		\draw[lightgray]  (0,0) grid (7,7);
		\draw[fill] (6,6) circle (2pt);
		\foreach \i in {(6,5),(5,4),(4,3),(3,2),(2,2)}
		{
			\draw[fill,white] \i circle (2pt);
			\draw \i circle (2pt);
		}
		\end{tikzpicture}
		\hspace{5mm}
		\begin{tikzpicture}[scale=0.5]
		\draw[lightgray]  (0,0) grid (7,7);
		\draw[fill] (6,6) circle (2pt);
		\foreach \i in {(5,5),(5,4),(5,3),(5,2),(6,2)}
		{
			\draw[fill,white] \i circle (2pt);
			\draw \i circle (2pt);
		}
		\end{tikzpicture}
	\end{center}
	\caption{The equality of the binomial coefficient represented by the  black point and  the sum of  binomial coefficients represented by white points is invariant by translations.}\label{fig:ex1}
\end{figure}
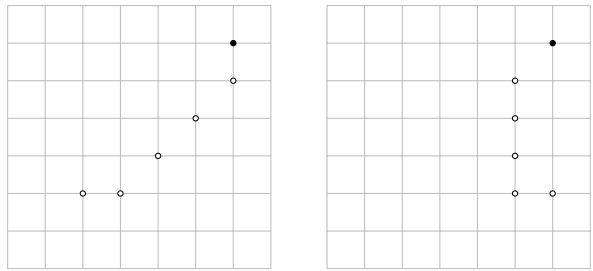

\subsection{Walls, rubble, and pavement}

A {\it wall} is a binomial sum which has precisely one binomial coefficient in each \emph{diagonal} of a set of nonnegative consecutive  diagonals (the binomial coefficient $\binom{n}{k}$ belongs to the diagonal $n-k$, see Section~\ref{sec:geom_repr}). For example, the right--hand side of \eqref{eq:vertical} excluding the last binomial coefficient is a wall.
Let ${\ell}$ be a non-negative integer and let $\w=(w_0,w_1,\ldots ,w_h)$ be a nonincreasing sequence of integers with $h\le \ell$. The {\it wall} $(\w,\ell)$ is the binomial sum
$$
(\w,\ell)=\sum_{i=0}^h {w_i+\ell-i\choose w_i},
$$ 
which has one binomial coefficient in each of the diagonals $\ell-h,\ell-h+1,\cdots, \ell$. The walls have the lex ordering between them, but this identification between a wall and its sum of binomial coefficients  allows for a natural comparaison between the walls and binomial expressions. (A $k$-binomial expression is a sum of binomial coefficients and it is not necessarily a $k$-binomial decomposition.)

We say that the wall $(\w,\ell)=\sum_{i=0}^h {w_i+\ell-i\choose w_i}$ {\it dominates} a $k$--binomial expression ${\b \choose k}=\sum_{i=0}^t{b_i\choose k-i}$ if the following three conditions hold:
\begin{enumerate}[label=(D\arabic*)]
\item \label{en:w1}  
$b_0-k\le \ell$
\item \label{en:w2} 
if $w_i=k-j\ge k-t$ then  $b_j< w_i+\ell-i$
\item \label{en:w3} $w_0\leq k$
\end{enumerate}
We write ${\b\choose k}\preceq (\w,\ell)$ to denote that $(\w,\ell)$ dominates ${\b \choose k}$. In this case, by \ref{en:w1}, no term  in the binomial expression lies in a diagonal higher than the first term in $(\w,\ell)$  and, by \ref{en:w2}, if there is a term in $(\w,\ell)$ in the same vertical line than a term in ${\b\choose k}$ then the latter lies on a strictly lower diagonal  than the former. Furthermore, by \ref{en:w3} the first column of the wall is more to the left (see Section~\ref{sec:geom_repr}) than the first column of the $k$-binomial expression. Overall, if a wall cominates a binomial expression, then the wall is to the top left of the binomial expression, in the geometric representation following Section~\ref{sec:geom_repr}.

A binomial sum of the form $R=\sum_{i\in I} \alpha_i{x_i\choose 0}$, where the binomial denominator of all the binomial coefficients are $0$, is called \emph{rubble}.

A binomial sum of the form $C=\sum_{i\in I} \alpha_i{i-1\choose i}$, where the binomial denominator of all the binomial coefficients is one more than the binomial numerator, is said to be a \emph{pavement}.

Given two sequences $\b$ and $\c$ of length $\leq t$, we let
\begin{align*}
\max\{\b,\c\}&=(\max\{b_0,c_0\},\ldots,\max\{b_t,c_t\}) \nonumber \\ \min\{\b,\c\}&=(\min\{b_0,c_0\},\ldots,\min\{b_{t'},c_{t'}\}) \nonumber 
\end{align*}
 where $t'$ is the minimal index for which both and $b_{t'}$ and $c_{t'}$ exist.
\subsection{Sketch of the proof of Lemma \ref{lem:abck} } \label{subsec:sketch}

The proof of Lemma \ref{lem:abck}, \eqref{eq:abc1} and \eqref{eq:abc11},  is by induction on $k$ and $m$. Let us sketch the case for  \eqref{eq:abc1} (\eqref{eq:abc11} is done similarly).
\begin{description}
	\item[Step 1:] The base cases for the induction are for $k=1$ and all $m$, and for $m\le k$ and all $k$.
	\item[Step 2:] The induction step assumes $m>k>1$. Then we argue that 
	$\ell(\b)\leq k'$ and $\ell(\c)\leq k''$ and that they are binomial 
	decompositions can be assumed by the nature of the statements 
	(see Claim~\ref{cl.3}). 
	\item[Step 3 (key step):] Then we find
 a nonzero binomial sum $S$, a rubble $R$, and a pavement $C$ where each term has value zero (each term ${i-1\choose i}$ satisfies $0<i$), such that the following translation invariant identities hold:
	\begin{align}
	{\a\choose k}&\doteq {\a'\choose k}+S+R\label{eq:asr}\\
	{\b\choose k'}+{\c\choose k''}&\doteq {\b'\choose k'}+{\c'\choose k''}+S\label{eq:bcs}+C,
	\end{align}
	for some $k$-binomial decomposition $\a'<_{lex} \a$, and $\b'\le_{lex} \b$ and $\c'\le_{lex} \c$ where $\b'$ and $\c'$ are $k'$ and $k''$ binomial decompositions (perhaps being the empty ones).
	\item[Step 4 (small touches):] Then using that $C$ is a binomial sum representing the zero, we obtain the equality (removing  $S$ in both sides)
	\begin{equation}\label{eq:rubble}
	 {\a'\choose k}+R={\b'\choose k'}+{\c'\choose k''},
	\end{equation}
	and, as $R\geq0$, we let
	\begin{equation} \label{eq:where_to_apply_ind}
	{\a'\choose k}={\b''\choose k'}+{\c''\choose k''}
	\end{equation}
	with $\b''\leq_{\text{lex}} \b'$ and $\c''\leq_{\text{lex}} \c'$ (prioritizing reducing $\c'$ before $\b'$ so that $\b''\geq_{\text{lex}}\a'-1$ if such condition is required).
	\item[Step 5 (induction):] If $S\neq \emptyset$, then ${\b''\choose k'}+{\c''\choose k''}\leq {\b'\choose k'}+{\c'\choose k''}<{\b\choose k'}+{\c\choose k''}$. If $S=\emptyset$ then $R\neq \emptyset$ and ${\b''\choose k'}+{\c''\choose k''}<{\b'\choose k'}+{\c'\choose k''}={\b\choose k'}+{\c\choose k''}$. In both cases, we use the induction hypothesis on \eqref{eq:where_to_apply_ind} to conclude
	\begin{equation}\label{eq:abck-1}
	{\a'\choose k-1}\le {\b''\choose k'-1}+{\c''\choose k''-1},
	\end{equation}
\item[Step 6 (concluding argument):]	By using the translation invariance of \eqref{eq:asr} and \eqref{eq:bcs} we obtain
	\begin{align*}
	{\a\choose k-1}&\doteq{\a'\choose k-1}+S'+R'\\
	{\b\choose k'-1}+{\c\choose k''-1}&\doteq{\b'\choose k'-1}+{\c'\choose k''-1}+S'+C',
	\end{align*}
	(where $S'$, $C'$, and $R'$ are obtained, respectively, from $S$, $C$ and $R$ by substracting one unit to each binomial denominator). Now combining the above with \eqref{eq:abck-1}, and $C'\geq 0$, and $R'=0$, yields that
	\begin{align}
	{\a\choose k-1}&\doteq{\a'\choose k-1}+S'+R' \nonumber \\
	&={\a'\choose k-1}+S' \qquad \text{as $R'=0$, since $R$ is rubble}\nonumber \\
	&\leq {\b''\choose k'-1}+{\c''\choose k''-1}+S' \qquad \text{by \eqref{eq:abck-1}}  \nonumber \\
	&\leq {\b'\choose k'-1}+{\c'\choose k''-1}+S' \qquad \text{as $\b''\leq_{\text{lex}} \b'$ and $\c''\leq_{\text{lex}} \c'$} \nonumber \\
	&\leq {\b'\choose k'-1}+{\c'\choose k''-1}+S'+C' \qquad \text{$C'\geq 0$} \nonumber \\
	&\doteq {\b\choose k'-1}+{\c\choose k''-1}  \qquad \text{by translation invariance}
	\end{align}
	 which gives the desired inequality \eqref{eq:abc1}.
\end{description}    

\noindent The detail of the implementation of this strategy can be found in the following subsections. The most challenging condition is to find the $R$, $S$, $C$ with the desired properties. We use Lemma~\ref{lem:rec_const} below to perform such reduction, with some additional case analysis.

The reduction from Lemma~\ref{lem:rec_const} is based on pushing the elements of $\b$, $\c$ and a part of $\a$ (forming a wall) away from each other, pushing the $\b$ towards the bottom border (coefficients of the type ${i-1\choose i}$ with $i>0$ grouped in the pavement $C$), and the $\a$ towards the left-hand-side vertical line (coefficients of the type ${i\choose 0}$ grouped in the rubble); when $\b$ or $\c$ collide with the wall, they are grouped in $S$.

The case analysis is done as the induction jumps through the sequence $\a$ substracting one unit to one of its elements not necessarily the last one (the case we are substracting one to ${a_{k-1}\choose 1}$ is treated separately). We also should control the condition $\b'\geq_{\text{lex}} \a''-1$, which is handled ``externally'' to the reduction argument.



\subsection{Reducing the wall and pair of 
	\texorpdfstring{$k$}{k}-binomial decompositions} \label{obs:rrp_2}

\begin{lemma}[Recursive reduction] \label{lem:rec_const}
Let $k\geq 1$ be an integer, $(\w,\ell)$ be a wall of height $h$, $\b$ and $\c$ be $k$-binomial decompositions.
	
	Assume $\max \{\b,\c\}=\b$, $\min \{\b,\c\}=\c$,  $\b$ and $(\w,\ell)$ are not the empty set, ${\b \choose k} \preceq (\w,\ell)$, ${\c \choose k} \preceq (\w,\ell)$, and $w_h\geq 1$.
	Then, 
	\begin{enumerate}[label=(\roman*)]
		\item \label{en:thesis_1} there exist: a wall $(\w',\ell)$, $k$-binomial decomposition $\b'$ and $\c'$, a rubble $R=\sum_{i\in I} {i\choose 0}$, a pavement $C=\sum_{i\in J} {i-1\choose i}$ with $j\geq 1$ for all $j$ in the multiset of integers $J$, and a sum of binomial coefficients $S$ such that
		\[
		{\b \choose k}+{\c \choose k}\doteq {\b' \choose k}+{\c' \choose k}+C+S
		\]
		and 
		\[
		(\w,\ell)\doteq (\w',\ell) + R + S
		\]
		with $\w'<_{\text{lex}} \w$ and $w_i'\leq w_i$ for each $i$ (for which it makes sense),  ${\b' \choose k} \preceq (\w',\ell)$, ${\c' \choose k} \preceq (\w',\ell)$, $b_0'\leq b_0$, $c_0'\leq c_0$, ${\b' \choose k}+{\c' \choose k}< {\b \choose k}+{\c \choose k}$.
		\item \label{en:thesis_2} Furthermore, $\max \{\b',\c'\}=\b'$, $\min \{\b',\c'\}=\c'$.
		\item \label{en:thesis_3} If, additionally, $b_0-k\geq \ell$, then \ref{en:thesis_1} holds, and $(\w',\ell)$ is the empty wall.
		\item\label{en:thesis_4} If, additionally, $w_0\geq k$, then \ref{en:thesis_1} holds and $\b'$ is the empty sequence and $\c'=\c$.
	\end{enumerate}
\end{lemma}

Before proceeding further, let us show the immediate consequence of Lemma~\ref{lem:rec_const} that we use in the arguments.

\begin{corollary}[Recursive Reduction, final form] \label{cor:rec_const_fin}
With the same assumptions as Lemma~\ref{lem:rec_const},
 there exist  a wall $(\w',\ell)$, $k$-binomial decompositions  $\b'$ and $\c'$, a rubble $R=\sum_{i\in I} {i\choose 0}$, a pavement $C=\sum_{i\in J} {i-1\choose i}$ with $j\geq 1$ for all $j$ in the multiset of integers $J$, and a sum of binomial coefficients $S$ such that
		\[
		{\b \choose k}+{\c \choose k}\doteq {\b' \choose k}+{\c' \choose k}+C+S
		\]
		and 
		\[
		(\w,\ell)\doteq (\w',\ell) + R + S
		\]
		with either $(\w',\ell)$ the empty wall, or both $\b'$ and $\c'$ being the empty sequences.
\end{corollary}

\begin{proof}[Proof of Corollary~\ref{cor:rec_const_fin}]
Apply Lemma~\ref{lem:rec_const} to $(\w,\ell)$, $\b$, $\c$ and obtain $[(\w',\ell),\b',\c',R,S,C]$. Now,
\begin{itemize}
	\item If Lemma~\ref{lem:rec_const} part \ref{en:thesis_1}+\ref{en:thesis_2} are satisfied, then we apply again Lemma~\ref{lem:rec_const} with
	$k\leftarrow k$, $(\w,\ell)\leftarrow (\w',\ell)$, $\b\leftarrow \b'$, and $\c\leftarrow \c'$ with strictly smaller $\w'<_{\text{lex}} \w$ and ${\b' \choose k}+{\c' \choose k}< {\b \choose k}+{\c \choose k}$.
	\item If Lemma~\ref{lem:rec_const} part \ref{en:thesis_3} is satisfied, then we are done.
	\item If Lemma~\ref{lem:rec_const} part \ref{en:thesis_4} is satisfied, and $(\w,\ell)$ is not the empty wall, then we apply  Lemma~\ref{lem:rec_const} again with $k\leftarrow k$, $(\w,\ell)\leftarrow (\w',\ell)$, $\b\leftarrow \c'$, and $\c\leftarrow \emptyset$.
	\item If $(\w,\ell)$ is the empty set, then the statement is satisfied.
	\item If $\b$ and $\c$ are the empty set, then the statement is satisfied.
\end{itemize}
Thus, repeating one of the points above, and as we are always reducing either $(\w,\ell)$, $\b$ or $\c$ we finish in one of the conditions given in the statement of the corollary.
	\end{proof}

\begin{proof}[Proof of Lemma~\ref{lem:rec_const}]
Observe that we are actually showing three results, \ref{en:thesis_1}+\ref{en:thesis_2}, \ref{en:thesis_3}, and \ref{en:thesis_4}. All the results are shown by induction on $\w$, and ${\b \choose k}+{\c \choose k}$, with the base cases begin when $(\w,\ell)$ is the empty set, or $\b$ (and hence $\c$) being the empty set. In those cases, the respective $\b'\leftarrow \b$, $\c'\leftarrow \c$ and $(\w',\ell)\leftarrow (\w,\ell)$ satisfy the conditions.

Let $j=b_t-(k-t)$ be the diagonal where the last term of ${\b\choose k}$ lies on. We perform a case analysis in $j$.



\noindent\framebox[1.1\width][l]{\textbf{Case 1: $\ell-h\leq j \leq \ell$.}}
In this case the ``$\b$ hits the wall''.
Some case analysis.

\noindent\framebox[1.1\width][l]{\textbf{Case 1.1: If $k-t<k$.}} We use \ref{en:thesis_3} on
$k\leftarrow k-t$
$\b\leftarrow (b_t)$, $\c\leftarrow \emptyset$, $(\w,\ell)\leftarrow ((w_{\ell-j},\ldots,w_h),j)$, to obtain 
$\b''$, $\c''$, $(\w'',j)$ as the empty set, $R''$, $S''$, $C''$.

To show \ref{en:thesis_4}, we apply \ref{en:thesis_4} again with
$\b \leftarrow (b_0,\ldots,b_{k-t-1},\b'')$, $\c \leftarrow \c$, $(\w,\ell)\leftarrow ((w_0,\ldots,w_{\ell-j-1}),\ell)$, to obtain 
$\b'''$ the empty set, $\c'''$, $(\w''',\ell)$, $R'''$, $S'''$, $C'''$.
Then \ref{en:thesis_4} follows with
$\b'\leftarrow \b'''$ the empty set, $\c'\leftarrow \c'''$, $(\w',\ell)\leftarrow (\w''',\ell)$, $R\leftarrow R'''+R''$, $S\leftarrow S'''+S''$, $C\leftarrow C'''+C''$.

To show [\ref{en:thesis_1}+\ref{en:thesis_2}, \ref{en:thesis_3}] we apply 
[\ref{en:thesis_1}+\ref{en:thesis_2}, \ref{en:thesis_4}] inductively to
$\b \leftarrow (b_0,\ldots,b_{k-t-1},\max\{\b'',\c''\})$, $\c \leftarrow (c_0,\ldots,c_{t-1},\min\{\b'',\c''\})$, $(\w,\ell)\leftarrow ((w_0,\ldots,w_{\ell-j-1}),\ell)$, to obtain 
$\b'''$, $\c'''$, $(\w''',\ell)$, $R'''$, $S'''$, $C'''$.
Then [\ref{en:thesis_1}+\ref{en:thesis_2}, \ref{en:thesis_3}] follows with
$\b'\leftarrow \b'''$, $\c'\leftarrow \c'''$, $(\w',\ell)\leftarrow (\w''',\ell)$, $R\leftarrow R'''+R''$, $S\leftarrow S'''+S''$, $C\leftarrow C'''+C''$.

\noindent \framebox[1.1\width][l]{\textbf{Case 1.2: If $k-t=k$.}}
Then $b_t-(k-t)=b_0-k=\ell$, and $b_0-(k-w_{0})=w_0+\ell$.
 The condition ${\b \choose k} \preceq (\w,\ell)$ prevents this case from occuring when showing \ref{en:thesis_4}.
 
We let 
\[
{b_0\choose k}\doteq \sum_{i=0}^{k-w_{0}-1}{b_0-i-1 \choose k-i}+{b_0-(k-w_{0}) \choose w_{0}}.
\]
There is a further case analysis.

\noindent \framebox[1.1\width][l]{\textbf{Case 1.2.1: If $b_0>k$.}}

Then we inductively apply \ref{en:thesis_3} (as we are in this case when $k-t=k$), again with
$\b\leftarrow (b_0-1,b_0-2,\ldots,b_{0}-(k-w_{0}))$, $\c\leftarrow \emptyset$, $(\w,\ell'')\leftarrow ((w_{1},\ldots,w_{h}),\ell-1)$ and obtain
$\b''$, $\c''$ the empty set, $(\w'',\ell'')$ the empty set, $R''$, $S''$, $C''$.

Then \ref{en:thesis_3} follows with
$\b'\leftarrow \b'''$, $\c'\leftarrow \c$, $(\w',\ell)$ the empty wall, $R\leftarrow R'''+R''$, $S\leftarrow S'''+S''+{b_0-(k-w_{0}) \choose w_{0}}$, $C\leftarrow C'''+C''$.

Then \ref{en:thesis_1}+\ref{en:thesis_2} follows with
$\b'\leftarrow \max\{\b''',\c\}$, $\c'\leftarrow \min\{\b''',\c\}$, $(\w',\ell)$ the empty wall, $R\leftarrow R'''+R''$, $S\leftarrow S'''+S''+{b_0-(k-w_{0}) \choose w_{0}}$, $C\leftarrow C'''+C''$.

\noindent \framebox[1.1\width][l]{\textbf{Case 1.2.2: If $b_0=k$.}}
Then $\ell=0$, and there is no more wall after removing the term ${b_0-(k-w_{0}) \choose w_{0}}$. 

Then 
 \ref{en:thesis_3} follows with
$\b'$ the empty sequence, $\c'\leftarrow \c$, $(\w',\ell)$ the empty wall, $R\leftarrow \emptyset$, $S\leftarrow {b_0-(k-w_{0}) \choose w_{0}}$, $C \allowbreak\leftarrow \allowbreak \sum_{i=0}^{k-w_{0}-1} \allowbreak{b_0-i-1 \choose k-i}$.

Then 
\ref{en:thesis_1}+\ref{en:thesis_2} follows with
$\b'\leftarrow \c$, $\c'$ the empty sequence, $(\w',\ell)$ the empty wall, $R\leftarrow \emptyset$, $S\leftarrow {b_0-(k-w_{0}) \choose w_{0}}$, $C \leftarrow \sum_{i=0}^{k-w_{0}-1}{b_0-i-1 \choose k-i}$.

\noindent\framebox[1.1\width][l]{\textbf{Case 2: $0\leq j<\ell-h$.}}
In this case the ``wall is above $\b$''.


\noindent \framebox[1.1\width][l]{\textbf{Case 2.1:  $w_h> k-t$.}}
So $\b$ streches strictly beyond the vertical given by $w_h$. Use the column of $w_h$ to kill the corresponding part of $\b$, then add the part of $c$ and continue.
That is to say.

We use \ref{en:thesis_3} on
$k\leftarrow w_h$
$\b\leftarrow (b_{k-w_h},\ldots,b_t)$, $\c\leftarrow \emptyset$, $(\w,\ell-h)\leftarrow ((w_{h}),\ell-h)$, to obtain 
$\b''$  as the empty set, $\c''$ as the empty set, $(\w'',\ell-h)$, $R''$, $S''$, $C''$.

To show \ref{en:thesis_4}, we apply \ref{en:thesis_4} with
$\b \leftarrow (b_0,\ldots,b_{k-t-1},\b'')$, $\c \leftarrow \c$, $(\w,\ell)\leftarrow ((w_0,\ldots,w_{\ell-j-1}),\ell)$, to obtain 
$\b'''$ the empty set, $\c'''$, $(\w''',\ell)$, $R'''$, $S'''$, $C'''$.
Then \ref{en:thesis_4} follows with
$\b'\leftarrow \b'''$ the empty set, $\c'\leftarrow \c'''$, $(\w',\ell)\leftarrow (\w''',\ell)$, $R\leftarrow R'''+R''$, $S\leftarrow S'''+S''$, $C\leftarrow C'''+C''$.

To show [\ref{en:thesis_1}+\ref{en:thesis_2}, \ref{en:thesis_3},\ref{en:thesis_4}] we apply 
[\ref{en:thesis_1}+\ref{en:thesis_2},\ref{en:thesis_3}, \ref{en:thesis_4}] inductively to
$\b \leftarrow (b_0,\ldots,b_{w_h-1},c_{w_h},\ldots,c_{t'})$, $\c \leftarrow (c_0,\ldots,c_{w_h-1})$, $(\w,\ell)\leftarrow ((w_0,\ldots,w_{h-1},\w''),\ell)$, to obtain 
$\b'''$, $\c'''$, $(\w''',\ell)$, $R'''$, $S'''$, $C'''$.
Then [\ref{en:thesis_1}+\ref{en:thesis_2}, \ref{en:thesis_3},\ref{en:thesis_4}] follows with
$\b'\leftarrow \b'''$, $\c'\leftarrow \c'''$, $(\w',\ell)\leftarrow (\w''',\ell)$, $R\leftarrow R'''+R''$, $S\leftarrow S'''+S''$, $C\leftarrow C'''+C''$.

\noindent \framebox[1.1\width][l]{\textbf{Case 2.2:  $w_h=k-t$.}} So $\b$ streches right until the vertical of the last part of the wall. We use $w_h$ to explicitly remove $b_{t}$ and continue.

Let 
\begin{equation} \label{eq:decomp_hor0}
{w_h+\ell-h\choose w_h}\doteq \sum_{i=0}^{w_h+\ell-h-b_t-1}{w_h+\ell-h-i-1 \choose w_h-1}+{b_t \choose w_{h}}
\end{equation}

\noindent \framebox[1.1\width][l]{\textbf{Case 2.2.1:  $w_h>1$.}}
To show [\ref{en:thesis_1}+\ref{en:thesis_2}, \ref{en:thesis_3},\ref{en:thesis_4}] we apply 
[\ref{en:thesis_1}+\ref{en:thesis_2},\ref{en:thesis_3}, \ref{en:thesis_4}] inductively to
$\b \leftarrow (b_0,\ldots,\allowbreak b_{t-1},c_{t})$, $\c \leftarrow (c_0,\ldots,c_{t-1})$, $(\w,\ell)\leftarrow ((w_0,\ldots,w_{h-1},w_h-1,\ldots,w_h-1),\ell)$ ($w_h+\ell-h-b_t$ copies of $w_h-1$), to obtain 
$\b'''$, $\c'''$, $(\w''',\ell)$, $R'''$, $S'''$, $C'''$.
Then [\ref{en:thesis_1}+\ref{en:thesis_2}, \ref{en:thesis_3},\ref{en:thesis_4}] follows with
$\b'\leftarrow \b'''$, $\c'\leftarrow \c'''$, $(\w',\ell)\leftarrow (\w''',\ell)$, $R\leftarrow R'''$, $S\leftarrow S'''+{b_t \choose w_{h}}$, $C\leftarrow C'''$.

\noindent \framebox[1.1\width][l]{\textbf{Case 2.2.2:  $w_h=1$.}} 
To show [\ref{en:thesis_1}+\ref{en:thesis_2}, \ref{en:thesis_3},\ref{en:thesis_4}] we apply 
[\ref{en:thesis_1}+\ref{en:thesis_2},\ref{en:thesis_3}, \ref{en:thesis_4}] inductively to
$\b \leftarrow (b_0,\ldots,\allowbreak b_{t-1},c_{t})$, $\c \leftarrow (c_0,\ldots,c_{t-1})$, $(\w,\ell)\leftarrow ((w_0,\ldots,w_{h-1}),\ell)$, to obtain 
$\b'''$, $\c'''$, $(\w''',\ell)$, $R'''$, $S'''$, $C'''$.
Then [\ref{en:thesis_1}+\ref{en:thesis_2}, \ref{en:thesis_3},\ref{en:thesis_4}] follows with
$\b'\leftarrow \b'''$, $\c'\leftarrow \c'''$, $(\w',\ell)\leftarrow (\w''',\ell)$, $R\leftarrow R'''+\sum_{i=0}^{w_h+\ell-h-b_t-1}{w_h+\ell-h-i-1 \choose w_h-1}$, $S\leftarrow S'''+{b_t \choose w_{h}}$, $C\leftarrow C'''+C''$.

\noindent \framebox[1.1\width][l]{\textbf{Case 2.3:  $w_h<k-t$.}}
Extend $b_t$ until beneath $w_h$, which will be above the extended part of $b_t$; then erase that part with $w_h$ using \ref{en:thesis_3}, and continue. this involve a case analysis depending on the leftover part of $b_t$.

Let 
\begin{equation} \label{eq:decomp_hor1}
{b_{t}\choose k-t}\doteq \sum_{i=0}^{k-w_{h}-1}{b_t-i-1 \choose k-i}+{b_t-(k-w_{h}) \choose w_{h}}
\end{equation}

We use \ref{en:thesis_3} on
$k\leftarrow w_h$
$\b\leftarrow (b_t-(k-w_{h}))$, $\c\leftarrow \emptyset$, $(\w,\ell-h)\leftarrow ((w_{h}),\ell-h)$, to obtain 
$\b''$  as the empty set, $\c''$ as the empty set, $(\w'',\ell-h)$, $R''$, $S''$, $C''$.

Note that $S''$ will be ${b_t-(k-w_{h}) \choose w_{h}}$, as we will use case 2.2 in the following step.
Now there is a case analysis, as we should decide where to put the leftover from the decomposition of ${\b_t\choose k-t}$ from \eqref{eq:decomp_hor1}.

\noindent \framebox[1.1\width][l]{\textbf{Case 2.3.1: $b_t>k$.}}
To show [\ref{en:thesis_1}+\ref{en:thesis_2}, \ref{en:thesis_3},\ref{en:thesis_4}] we apply 
[\ref{en:thesis_1}+\ref{en:thesis_2},\ref{en:thesis_3}, \ref{en:thesis_4}] inductively to
$\b \leftarrow (b_0,\ldots,\allowbreak b_{t-1},\max\{(b_t-1,\ldots,b_t-(k-w_h-1)-1),(c_t)\})$, $\c \leftarrow (c_0,\ldots,c_{t-1},\max\{(b_t-1,\ldots,b_t-(k-w_h-1)-1),(c_t)\})$, $(\w,\ell)\leftarrow ((w_0,\ldots,w_{h-1},\w''),\ell)$, to obtain 
$\b'''$, $\c'''$, $(\w''',\ell)$, $R'''$, $S'''$, $C'''$.
Then [\ref{en:thesis_1}+\ref{en:thesis_2}, \ref{en:thesis_3},\ref{en:thesis_4}] follows with
$\b'\leftarrow \b'''$, $\c'\leftarrow \c'''$, $(\w',\ell)\leftarrow (\w''',\ell)$, $R\leftarrow R'''+R''$, $S\leftarrow S'''+S''$, $C\leftarrow C'''+C''$.

\noindent \framebox[1.1\width][l]{\textbf{Case 2.3.2: $b_t=k$.}}
To show [\ref{en:thesis_1}+\ref{en:thesis_2}, \ref{en:thesis_3},\ref{en:thesis_4}] we apply 
[\ref{en:thesis_1}+\ref{en:thesis_2},\ref{en:thesis_3}, \ref{en:thesis_4}] inductively to
$\b \leftarrow (b_0,\ldots,\allowbreak b_{t-1},c_t)\})$, $\c \leftarrow (c_0,\ldots,c_{t-1})$, $(\w,\ell)\leftarrow ((w_0,\ldots,w_{h-1},\w''),\ell)$, to obtain 
$\b'''$, $\c'''$, $(\w''',\ell)$, $R'''$, $S'''$, $C'''$.
Then [\ref{en:thesis_1}+\ref{en:thesis_2}, \ref{en:thesis_3},\ref{en:thesis_4}] follows with
$\b'\leftarrow \b'''$, $\c'\leftarrow \c'''$, $(\w',\ell)\leftarrow (\w''',\ell)$, $R\leftarrow R'''+R''$, $S\leftarrow S'''+S''$, $C\leftarrow C'''+C''+\sum_{i=0}^{k-w_{h}-1}{b_t-i-1 \choose k-i}$.
\end{proof}

In Figure~\ref{fig:reductions} we can find a geometric representation of the reductions performed in Lemma~\ref{lem:rec_const}.

\begin{figure}[ht]
	\begin{center}
		\begin{tikzpicture}[scale=0.5]
		\draw[lightgray]  (0,0) grid (7,7);
		\draw[dotted] (0,2)--(5,7);
		\draw[fill,blue] (5,7) circle (2pt);
		\foreach \i in {(5,6),(4,5),(3,5)}
		{
			\draw[fill,white] \i circle (4pt);
			\draw[blue] \i circle (4pt);
		}
		\draw[fill, red] (3,5) circle (2pt);
		\node at (3,-1) {(1.a)};
		\end{tikzpicture}
		\hspace{5mm}
		\begin{tikzpicture}[scale=0.5]
		\draw[lightgray]  (0,0) grid (7,7);
		\draw[dotted] (0,2)--(5,7);
		\draw[dotted] (0,1)--(6,7);
		\draw[dotted] (0,0)--(7,7);
		\draw[fill,blue] (5,7) circle (2pt);
		\foreach \i in {(5,6),(4,4),(3,2), (3,5), (2,3), (2,2)}
		{
			\draw[fill,white] \i circle (4pt);
			\draw[blue] \i circle (4pt);
		}
		\draw[fill, red] (3,5) circle (2pt);
		\draw[fill, red] (2,3) circle (2pt);
		\draw[fill, red] (2,2) circle (2pt);
		
		\node at (3,-1) {(1.b)};
		\end{tikzpicture}
		\hspace{5mm}
		\begin{tikzpicture}[scale=0.5]
		\draw[lightgray]  (0,0) grid (7,7);
		\draw[dotted] (0,2)--(5,7);
		\draw[dotted] (2,0)--(7,5);
		\draw[fill,blue] (6,4) circle (2pt);
		\draw[fill, red] (3,5) circle (2pt);
		\foreach \i in {(6,3),(5,2),(4,1), (3,1)}
		{
			\draw[fill,white] \i circle (4pt);
			\draw[blue] \i circle (4pt);
		}
		\foreach \i in {(2,4),(2,3),(2,2), (2,1), (3,1)}
		{
			\draw[fill,white] \i circle (3pt);
			\draw[red] \i circle (3pt);
		}
		\node at (3,-1) {(2.1)};
		\end{tikzpicture}
		\hspace{5mm}
		\begin{tikzpicture}[scale=0.5]
		\draw[lightgray]  (0,0) grid (7,7);
		\draw[dotted] (0,5)--(2,7);
		\draw[fill, red] (2,7) circle (2pt);
		\foreach \i in {(1,6),(1,5),(2,4), (0,3), (0,2), (0,1), (1,1)}
		{
			\draw[fill,white] \i circle (4pt);
			\draw[red] \i circle (4pt);
		}
		\draw[fill,blue] (2,4) circle (2pt);
		\draw[fill,blue] (1,1) circle (2pt);
		\node at (3,-1) {(2.2)};
		\end{tikzpicture}
	\end{center}
	\caption{The solid blue points are from ${\b\choose k}$, the solid red points are from $(\w,\ell)$, the hollow points are the result of the reductions in the various cases of Lemma \ref{lem:rec_const} and the bicolored points are the terms in $S$.}\label{fig:reductions}
\end{figure}
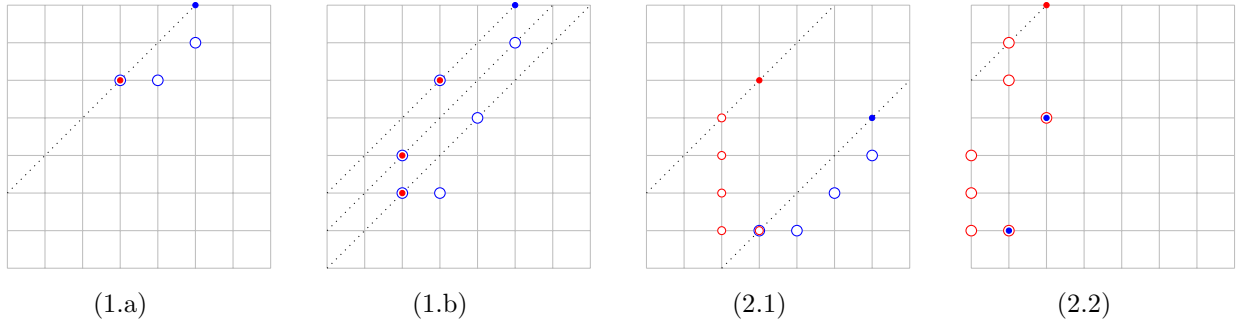

\begin{remark}
Lemma~\ref{lem:rec_const} is used in two places: the proof of  Lemma~\ref{lem:abck} and the proof of Corollary~\ref{cor.1}.
One could define the $n$-rubble ($n$ for negative) to be a sum of binomial terms of the form ${\cdot \choose -1}$. Then a statement similar to Lemma~\ref{lem:rec_const} follows with $w_h\geq 0$, and were the rubble $R$ is an $n$-rubble. The proof of such statement is similar as the argument leading to Lemma~\ref{lem:rec_const}, but the boundary cases for $w_h$ are $0$ instead of $1$. This implies that the cases 2.2.1 and 2.2.2 dissapear, and the new boundary cases occur in Case 2.3, where the separate case when $w_h=k-t-1=0$ should be distinguished; this is the one creating the rubble with binomial denominators $-1$. The rubble of the type ${\cdot \choose -1}$ very slightly simplifies the argument leading Lemma~\ref{lem:abck} and sketched in Section~\ref{subsec:sketch}. On the other hand, it gives difficulties in showing Corollary~\ref{cor.1}, as there we crucually use the properties of the rubble with ${\cdot \choose 0}$.
\end{remark}

\subsection{Proof of Lemma~\ref{lem:abck}}\label{subsec:proofabck}

The proof of  inequality \eqref{eq:abc1} in Lemma  \ref{lem:abck}  is by induction on $k$ and $m={\a \choose k}$.

\paragraph{Base cases 1: $k=1$ and any $m>0$.}

 For $k=1$ the result holds for each $m$. Indeed, we have ${\a\choose k-1}=1$. If $k',k''\geq k$, then at least of ${\b\choose k'}$ or ${\c\choose k''}$ is non-empty, and it contains a binomial coefficient that is larger or equal to ${i\choose i}$ with $i\geq 1$, hence ${\b\choose k'-1}+{\c\choose k''-1}\geq 1$. If $k'=1$ and $k''=k-1$, and ${\b\choose k}\geq 1$, then ${\b\choose k-1}\geq 1$ and the result also holds (as ${\c\choose k-2}$ is $\geq 0$ if ${\c\choose k-1}$ is non-empty and $=0$ if it is the empty sequence). If $k'=1$ and $k''=k-1$, and ${\b\choose k}=0$, then the condition $\b\geq_{\text{lex}}\a-1$ forces $b_0=k-1$, which implies that ${\b\choose k-1}\geq 1$, and thus shows the statement in these first base cases.

\paragraph{Base cases 2: any $k>1$ and $m\leq k$.}



We have
\[{\a\choose k}=\sum_{i=0}^{m-1}{k-i\choose k-i}, \] 
 \[{\b\choose k'}=\sum_{i=0}^{m'-1}{k'-i\choose k'-i}+\sum_{i=m'}^{m'''-1}{k'-i-1\choose k'-i}\]
  and 
  \[{\c\choose k''}=\sum_{i=0}^{m''-1}{k''-i\choose k''-i}+\sum_{i=m''}^{m''''-1}{k''-i-1\choose k''-i}\] while
\[{\a\choose k-1}=\sum_{i=0}^{m-1}{k-i\choose k-i-1},\]  
\[{\b\choose k'-1}=\sum_{i=0}^{m'-1}{k'-i\choose k'-i-1}+\sum_{i=m'}^{m'''-1}{k'-i-1\choose k'-i-1}\] and \[{\c\choose k''-1}=\sum_{i=0}^{m''-1}{k''-i\choose k''-i-1}+\sum_{i=m''}^{m''''-1}{k''-i-1\choose k''-i-1}\; . \]
The case $k',k''\geq k$ clearly holds as, for each $i$ for which $b_i\geq k'-i$ exists then ${b_i \choose k'-i-1}\geq {a_i \choose k-i-1}$ (and the same holds whenever $c_i\geq k''-i$), and there is the same number of terms in both places of binomial coefficients at both sides of the equality ${\a\choose k}={\b\choose k'}+{\c\choose k''}$.
To show the case $k'=k,k''=k-1$ we either have ${\b\choose k}>0$ and then we may use a similar argument as before, or we have ${\b\choose k}=0$. In this second case, the condition $\b\geq_{\text{lex}} \a-1$ imposed on the representation of zero implies that ${\b\choose k-1}\geq m$. Then, if $m\leq k-1$ we have ${\c\choose k-2}= (k-1)+(k-2)+\ldots+(k-m)$ while ${\a\choose k-1}=k+(k-1)+\ldots+(k-(m-1))$ and thus showing the result. The case when $m=k$ we have ${\a\choose k-1}=\binom{k+1}{2}$ while ${\c\choose k-2}=\binom{k}{2}$ and ${\b \choose k-1}\geq k$, thus showing the result as well in this case.

\paragraph{Induction step.}

Assume that $m>k>1$. We assume that $k'\ge k''$. We may also assume that
\begin{equation}\label{eq:assumption_diff}
(b_0,k')\neq (a_0, k)
\end{equation}
 since otherwise we may delete ${b_0\choose k'}$ and ${a_0\choose k}$  in \eqref{eq:abc0} and apply induction on $k$. Similarly, $(c_0,k'')\neq (a_0, k)$.
 
Using Claim~\ref{cl.3}, it suffices to show the statement when \eqref{eq:assumption_lem} is satisfied.

 Let $r$ be the largest subscript such that $a_r-(k-r)=a_0-k$. Thus ${a_r\choose k-r}$ corresponds to the  smaller binomial coefficient in ${\a\choose k}$ on the same diagonal as ${a_0\choose k}$. We write ${a_r\choose k-r}={a_r-1\choose k-r}+{a_r-1\choose k-r-1}$
and let $\a'$ be the sequence obtained from $\a$ by replacing $a_r$ by $a_r-1$. By the definition of $r$ and the fact that $m>k$, ${\a'\choose k}$ is  a $k$--binomial decomposition. 

Now there is a small case analysis on $r$. Observe that, as $\a$ is a $k$-binomial decomposition, $0\leq r\leq k-1$.

\noindent \textbf{Case $r=k-1$.} 
Then $k-r-1=0$. We have
\[
{\a \choose k}\doteq {\a'\choose k} + {a_r-1\choose k-r-1}\doteq {\a'\choose k} + {a_r-1\choose 0}\doteq {\a'\choose k} + {a_r-1\choose 0}\doteq{\a'\choose k} +R ={\a'\choose k} + 1
\]
Now, if $\c$ is a $k''$-binomial decomposition different from $0$, then we consider ${\c'\choose k''}\stackrel{b}{=}{\c\choose k''}-1$ and $\b'=\b$; if $\c$ is a $k''$-binomial decomposition equal to $0$, then we consider ${\b'\choose k'}\stackrel{b}{=}{\b\choose k'}-1$ and $\c'=\c$ the empty decomposition.
In both cases we have $\b'\leq_{\text{lex}} \b$, $\c'\leq_{\text{lex}} \c$, and binomial decompositions, and $\b'\geq_{\text{lex}} \a'-1$. Then we also have
\begin{align} 
{\a'\choose k}&={\b'\choose k'}+{\c'\choose k''}\label{eq:ind_hyp_1} \\
{\b'\choose k'-1}+{\c'\choose k''-1}&\leq {\b\choose k'-1}+{\c\choose k''-1} \nonumber \\
{\a'\choose k-1}&={\a\choose k-1} \nonumber
\end{align}
The last equality follows as ${\a\choose k-1}>k$,  and $r$ is on the largest diagonal and $r=k-1$, so $a_{k-1}-1\geq 1$, and we are only substracting $1$ unit to ${\a \choose k}$.
Since $\a'\leq_{\text{lex}}\a$, and is a $k$-binomial decomposition, and $\b'$ and $\c'$ are $k',k''$-binomial decompositions (and $\b'\geq_{\text{lex}}\a'-1$ in the $(k,k-1)$ case), by \eqref{eq:ind_hyp_1} we may apply the induction hypothesis on $\a',\b',\c'$ to conclude that
\[
{\a'\choose k-1}\leq {\b'\choose k'-1}+{\c'\choose k''-1}
\]
and then, if $R'={a_r-1\choose -1}=0$, we have:
\[
{\a\choose k-1}\doteq {\a'\choose k-1}+R'\leq {\b'\choose k'-1}+{\c'\choose k''-1}+R'={\b'\choose k'-1}+{\c'\choose k''-1}\leq {\b\choose k'-1}+{\c\choose k''-1}
\]
showing the claim in this case.

\noindent \textbf{Case $0\leq r<k-1$.} Then $k-r-1\geq 1$, and we consider the
wall $W={a_r-1\choose k-r-1}$ which satisfies $w_0=w_h=k-r-1\geq 1$.
By \eqref{eq:assumption_lem}, 
$\b,\c$ are $k'$ and $k''$ binomial decompositions.
Let us further assume that:
\begin{equation} \label{eq:assumption}
	\text{there exists the term $b_{k'-k+r+1}$ in $\b$ or the term $c_{k''-k+r+1}$ in $\c$.}
\end{equation}
The case when \eqref{eq:assumption} is not satisfied is given in the ``Otherwise'' case below.

Let $k'''=k-r-1$, and
lets define  ${\b^+\choose k'''}$ and ${\c^-\choose k'''}$  as 
\begin{align}
\b^+&=\max\{(b_{k'-k+r+1},b_{k'-k+r+1+1},\ldots,b_{k'-1}),(c_{k''-k+r+1},c_{k''-k+r+1},\ldots,c_{k''-1})\},\nonumber \\ \c^-&=\min\{(b_{k'-k+r+1},b_{k'-k+r+1+1},\ldots,b_{k'-1}),(c_{k''-k+r+1},c_{k''-k+r+1},\ldots,c_{k''-1})\} \nonumber
 \end{align}
 Using \eqref{eq:assumption}, the resulting binomial sums ${\b^+\choose k'''}$ and ${\c^-\choose k'''}$ are $k''$--binomial decompositions, and at least ${\b^+\choose k'''}$ is non-empty. We denote by 
\begin{align}
\b^{in}&=\begin{cases} (b_0,\ldots ,b_{k'-k+r}) & \text{if $b_{k'-k+r+1}$ exists} \\
\b& \text{otherwise}
\end{cases} \nonumber \\
\c^{in}&=\begin{cases} (c_0,\ldots ,c_{k''-k+r}) & \text{if $c_{k''-k+r+1}$ exists} \\
\c& \text{otherwise}
\end{cases} \nonumber
\end{align}
 the initial segments of the sequences $\b$ and $\c$, so that ${\b^{in}\choose k'}$ is a $k'$--binomial decomposition and ${\c^{in}\choose k''}$ is a $k''$--binomial decomposition
 \begin{equation}\label{eq:not_assumption}
\text{Note: if \eqref{eq:assumption} is not satisfied, then $\b=\b^{in}$ and $\c=\c^{in}$.}
\end{equation}
Since these simply involve a rearrangement of the binomial sums, we certainly have
\begin{equation}\label{eq:n+}
{\b^{in}\choose k'}+{\b^+\choose k'''}+{\c^{in}\choose k''}+{\c^-\choose k'''}\;\;\doteq\;\; {\b\choose k'} +{\c\choose k''}
\end{equation}

\begin{claim} \label{claim:1} The wall $W={a_r-1\choose k-r-1}$ (equivalently $((k-r-1),a_r-1-(k-r-1))$) dominates ${\b \choose k'}$ and ${\c \choose k''}$, and ${\b^+\choose k'''}$ and ${\c^{-}\choose k'''}$.
\end{claim}

\begin{proof} Since $k',k''\geq k'''$, and $k'''=k-r-1=w_0$, then condition \ref{en:w3} also holds for ${\b \choose k'}$, ${\c \choose k''}$, ${\b^+\choose k'''}$, ${\c^{-}\choose k'''}$. In particular, and since each binomial coefficient of ${\b^+\choose k''}$ and ${\c^{-}\choose k''}$ comes from ${\b\choose k'}$ and ${\c\choose k''}$, it suffices to show that $W$ dominates ${\b \choose k'}$ and ${\c \choose k''}$.	

We observe that ${b_0\choose k'}\le {\b\choose k'}\le {\a\choose k}<{a_0+1\choose k}$. Since binomial coefficients are increasing along diagonals and vertical lines and $k'\ge k$, then ${b_0\choose k'}$ lies on a diagonal not larger than $a_0-k$, and so do all terms of ${\b\choose k'}$. A similar argument applies to ${\c\choose k''}$ if $k''\ge k$. If $k''=k-1$ and $k'=k$ then \eqref{eq:assumption_lem} implies that $\c \le_{lex} \a-1$.   This proves condition \ref{en:w1} for ${\b \choose k'}$ and ${\c \choose k''}$.

Let us show that \ref{en:w2} holds. By \ref{en:w1}, all terms in ${\b\choose k'}$ and in ${\c\choose k''}$ lie in not larger diagonals than $a_0-k$.
Suppose that  ${b_0\choose k'}$ is in the same diagonal as ${a_0\choose k}$. Then $k'>k$ as we are assuming ${b_0\choose k'}\neq {a_0\choose k}$ (see \eqref{eq:assumption_diff}). Hence, ${b_0 \choose k'}>{a_0\choose k}$  and ${\b\choose k'}\le {\a\choose k}$ imply that, if there is a term of ${\b \choose k'}$ in the vertical line of ${a_r-1\choose k-r-1}$ then it should lie below it, as the tail $\sum_{i> r}{a_i\choose k-i}<{a_r-1\choose k-r-1}$ (note that $r$ is the largest index on the largest diagonal). If ${b_0\choose k'}$ is in a lower  diagonal than ${a_0\choose k}$ (which is the same as the diagonal of ${a_r-1\choose k-r-1}$), then a term of ${\b\choose k'}$ in the vertical line of  ${a_r-1\choose k-r-1}$, if exists, is below it. A similar argument applies to ${\c\choose k''}$ in the case that $k''\ge k$. In the case that $k''=k-1$ then $k'=k$ and as before \eqref{eq:assumption_lem} implies that $\c\le_{lex} \a-{\mathbf 1}$. It follows that a term of ${\c\choose k''}$ in the vertical line of ${a_j-1\choose k-j-1}$, if exists, is below it. This shows \ref{en:w2}.
\end{proof}

Using Claim~\ref{claim:1}, and the definition of ${\b^+\choose k'''}$ and ${\c^-\choose k'''}$, we conclude that the triple $\left[W,{\b^+\choose k'''},{\c^-\choose k'''}\right]$ satisfy the assumptions of Corollary~\ref{cor:rec_const_fin} with $k\leftarrow k''$, $(\w,\ell)\leftarrow ((k-r-1),a_r-1-(k-r-1))=W$ wall of height $0$, $\b\leftarrow \b^+$, and $\c\leftarrow \c^-$.
Hence, we obtain the pentuple
$\left(\left[W', {\b_1\choose k'''},{\c_1\choose k'''}\right], S_1, R_1, C_1\right)$
with
\begin{quote}
	$W'$ is the empty sum \qquad or,\qquad $\b_1$ and $\c_1$ are the empty sequences.
\end{quote}
Now we do a case analysis. We shall observe that, if $b_{k'-k+r+1}$ or $c_{k''-k+r+1}$ does not exists, then ${\c^-\choose k'''}$ is the empty sequence all along the process involved in Corollary~\ref{cor:rec_const_fin} (see the condition $c_0'\leq c_0$ in Part~\ref{en:thesis_1}).

\noindent\framebox[1.1\width]{\textbf{If $W'$ is the empty sum.}} Then the reduction process is finished. We let 
\[
{\b''\choose k'}=
\begin{cases}
{\b^{in}\choose k'}+{\b_1\choose k'''} & \text{if $b_{k'-k+r+1}$ exists and $(\b^{in},\b_1)$ is strictly decreasing} \\
{\b^{in}\choose k'}+{\c_1\choose k'''} & \text{if $b_{k'-k+r+1}$ exists and $(\b^{in},\b_1)$ is not strictly decreasing} \\
{\b^{in}\choose k'}&\text{ otherwise}
\end{cases}
\]
and
\[{\c''\choose k''}=
\begin{cases}
{\c^{in} \choose k''}+{\c_1\choose k'''} & \text{if $b_{k'-k+r+1}$ exists and $(\b^{in},\b_1)$ is strictly decreasing} \\
{\c^{in} \choose k''}+{\b_1\choose k'''} & \text{if $b_{k'-k+r+1}$ exists and $(\b^{in},\b_1)$ is not strictly decreasing} \\
{\c^{in}\choose k''}+{\b_1\choose k'''}&\text{ otherwise}
\end{cases}
\]
and $S=S_1$, $R=R_1$, $C=C_1$.
The argument for the case $k',k''\geq k$ continues below.

For the case $k'=k,k''=k-1$: if $b^{in}_{0}=b_{0}''$ does not exist correspond to cases with $m< k$ (by the condition $\b\geq_{\text{lex}} \a-1$), which are treated in the base cases.
Now, since $a_{r}'=a_r-1\geq k-r$ (otherwise $m< k$), we have that if $\b\geq_{\text{lex}}\a-1$, then $\b^{in}>_{\text{lex}}\a'-1$, as $b^{in}_r=b_r$ if $b_r$ exists, or there exists a smaller index $0\leq r'<r$ with $b_r=b^{in}_{r'}>a_{r'}=a'_{r'}$, and thus $\b^{in}>_{\text{lex}}\a'-1$ follows as well, which implies that $\b''>_{\text{lex}}\a'-1$.
%



\noindent \framebox[1.1\width]{\textbf{Otherwise: $\b_1$ and $\c_1$ are the empty sequences.}} We certainly have $W'<{\b^{in}\choose k'}+{\c^{in}\choose k''}$ (as $W<{\b^{in}\choose k'}+{\b^+\choose k'''}+{\c^{in}\choose k''}+{\c^-\choose k'''}$). We also have ${\b^{in}\choose k'}\preceq W'$ and ${\c^{in}\choose k'}\preceq W'$. This instance also handles the case when \eqref{eq:assumption} is not satisfied using \eqref{eq:not_assumption}. We apply the Corollary~\ref{cor:rec_const_fin} to $[W',{\c^{in}\choose k'},\emptyset]$ with return either $([\emptyset,{\c_2\choose k'},\emptyset],S_2,R_2, D_2)$ or  $[W'',\emptyset,\emptyset]$; in the first case we stop, and in the second we apply Corollary~\ref{cor:rec_const_fin} again to $[W'',{\b^{in}\choose k'},\emptyset]$ and will obtain $([\emptyset ,{\b_3\choose k'},\emptyset],S_3,R_3, D_3)$ (as $W'<{\b^{in}\choose k'}+{\c^{in}\choose k''}$ and the reduction of the wall leaves some non-zero summand, the ruble, while the reductions of $\b$ and $\c$ leave a non-zero summan, the pavement).
In the first case, we let $\b''=\b^{in}$, $\c''=\c_3$, $C=C_1+C_2$, $R=R_1+R_2$, $S=S_1+S_2$; in the second case, we let  $\b''=\b_3$, $\c''=\emptyset$, $C=C_1+C_2+C_3$, $R=R_1+R_2+R_3$, $S=S_1+S_2+S_3$.

\noindent \framebox[1.1\width]{\textbf{Resume the argument}}In both cases ($W'$ being empty, or $\b_1$ and $\a_1$ being empty), we can write the invariant identities
\begin{align}
{\b\choose k'}+{\c\choose k''}&\doteq{\b''\choose k'}+{\c''\choose k''}+S+C,\label{eq:invbc}\\
{\a\choose k}&\doteq {\a' \choose k}+{a_r-1\choose k-r-1}\doteq {\a' \choose k}+S+R.\label{eq:inva}
\end{align}
As $C=0$, it follows that
$$
{\a' \choose k}+R= {\b''\choose k'}+{\c''\choose k''},
$$

Now, let ${\b'''\choose k'}$ ${\c'''\choose k''}$ be two binomial decompositions such that 
$$
{\a' \choose k}= {\b'''\choose k'}+{\c'''\choose k''},
$$
and where $\b'''$ and $\c'''$ have been obtained by first decreasing $\c'''$
and then $\b'''$ if $\c'''$ turns to be the empty sequence. In particular, $\b'''\leq_{\text{lex}} \b''$ and $\c'''\leq_{\text{lex}} \c''$ which means that 
\begin{equation}\label{eq:ind_step_2}
{\b'''\choose k'-1}+{\c'''\choose k''-1}\leq {\b'''\choose k'-1}+{\c'''\choose k''-1}.
\end{equation}

 Overall, we have that either $R\neq \emptyset$ or $S\neq \emptyset$, which in particular implies that ${\b'''\choose k'}$ ${\c'''\choose k''}<{\b\choose k'}+{\c\choose k''}$.
Furthermore, in the case that $k'=k$ and $k''=k-1$ we also have
$\b''\geq _{\text{lex}} \a'-1$, and thus $\b'''\geq _{\text{lex}} \a'-1$
Therefore, by induction, 
\begin{equation}
\label{eq:ind_step}
{\a' \choose k-1}\le {\b'''\choose k'-1}+{\c'''\choose k''-1}.
\end{equation}
On the other hand, by translation invariance, \eqref{eq:invbc}, \eqref{eq:inva}, we conclude that
\begin{align}
{\b\choose k'-1}+{\c\choose k''-1}&\doteq {\b''\choose k'-1}+{\c''\choose k''-1}+S'+C',\label{eq:ind_step_4}\\
{\a\choose k-1}&\doteq {\a' \choose k-1}+S'+R', \label{eq:ind_step_3}
\end{align}
where $S'$, $C'$ and $R'$ are obtained by sustracting one unit to all lower terms of binomial coefficients in $S$, $C$ and $R$ respectively. As the pavement $C$ has the extra property that the terms ${i-1\choose i}$ satisfy $i\geq 1$, then $C'\geq 0$, and as all the binomial denominators of the terms in $R$ are $0$, then $R'=0$.
It follows that
\begin{align}
{\a\choose k-1}&\doteq {\a' \choose k-1}+S'+R' \qquad \eqref{eq:ind_step_3} \nonumber \\
&={\a' \choose k-1}+S' \qquad \text{as $R'=0$} \nonumber \\
&\le {\b'''\choose k'-1}+{\c'''\choose k''-1}+S' \qquad \text{by \eqref{eq:ind_step}} \nonumber \\
&\leq {\b'''\choose k'-1}+{\c'''\choose k''-1}+S'+C'  \qquad \text{$C'\geq 0$}
\nonumber \\
&\leq {\b''\choose k'-1}+{\c''\choose k''-1}+S'+C'  \qquad \eqref{eq:ind_step_2}
\nonumber \\
&\doteq{\b\choose k'-1}+{\c\choose k''-1}, \qquad \eqref{eq:ind_step_4}\label{eq:concl}
\end{align}
completing the proof for this general case.
\begin{observation}\label{obs:conseq_of_proof}
	If $C$ is not the empty set in this induction step, then there is an strict inequality in \eqref{eq:concl} and thus in \eqref{eq:abc1}.
\end{observation}


The proof of inequality \eqref{eq:abc11}  of Lemma~\ref{lem:abck}  follows with similar arguments as the one above for \eqref{eq:abc1}. Let us highlight the differences. The term $C'$ does not contribute to the inequality, unless the terms are of the type ${0\choose 1}$ in $C$, that turn into ${-1\choose 0}=1$ in $C'$; in particular, Observation~\ref{obs:conseq_of_proof} holds only in this case. $R'$ behaves as in the previous case. The base cases are similar: for $k=1$, as it only matters that we are achieving the terms with binomial denominator being $0$, the same argument holds. For the case $k>1$ and $m< k$, since we have the same number of terms in both sides, and they continue to be of the type ${i\choose i}$, \eqref{eq:abc11} holds; this argument also can be extended to $m=k$.


\subsection{Proof of Corollary~\texorpdfstring{\ref{cor.1}}{11}, or \texorpdfstring{\eqref{eq:abc-i}}{(5)} of Lemma \texorpdfstring{\ref{lem:abc}}{2}} \label{sec:proof_equality}


Let us first show  \eqref{eq:ai} and \eqref{eq:ai_m_1}. 

Now:
\begin{enumerate}
	\item Use Claim~\ref{cl.3} on the given $(\a,\b,\c)$ to obtain $(\a,\b'',\c'')$ with $\b''$ and $\c''$ being binomial decompositions.
	\item \label{en:red_2}  Reduce the $\c''$ to $\c'''$, if needed, to obtain an equality ${\a \choose k}={\b''\choose k}+{\c''' \choose k-1}$. If $\c'''$ turn into zero, then $\b''=\a$ and the results are obtained.
	\item \label{en:red_3} Apply Lemma~\ref{lem:abck} to $(\a,\b'',\c''')$
	\item \label{en:red_4} Use Observation~\ref{obs:minus1} on the result to obtain new $k-1$, $k-1$, and $k-2$ binomial decompositions $(\a',\b''',\c'''')$.
	\item \label{en:red_5} Repeat this procedure beginning at Step~\ref{en:red_2}, until the first binomial decomposition is comprised of a ${\cdot \choose 0}$ term; at this point, the argument is finished and the inequalities \eqref{eq:ai} and \eqref{eq:ai_m_1} are shown.
\end{enumerate}

Let us now move on towards the ``moreover'' part. The hypothesis states that we have equalities for the quadruples: $(\a,\b,\c,k)$ and $(\a,\b,\c,k-1)$.
That is,
\[
{\a \choose k}={\b \choose k}+{\c \choose k-1}, {\a \choose k-1}={\b \choose k-1}+{\c \choose k-2}
\]
then we consequently have
\[
{\a+1 \choose k}={\b+1 \choose k}+{\c+1 \choose k-1},
\]
and $\a+1,\b+1,\c+1$ are all binomial decompositions.
Then we have, by \eqref{eq:ai_m_1} applied to $(\a+1,\b+1,\c+1,k)$ with $i=2$, and \eqref{eq:ai_m_1} to $(\a,\b,\c,k-1)$, which are also binomial decompositions for such ``$k$'', and obtain that:
\[
{\a-1 \choose k-2}\leq {\b-1 \choose k-2}+{\c \choose k-3}, {\a-1 \choose k-1}\leq {\b-1 \choose k-1}+{\c \choose k-2}
\]
with the previous equality 
\[{\a \choose k-1}={\b \choose k-1}+{\c \choose k-2}\]
we conclude that, actually,
\[
{\a-1 \choose k-2}= {\b-1 \choose k-2}+{\c \choose k-3}, {\a-1 \choose k-1}= {\b-1 \choose k-1}+{\c \choose k-2}
\]
so we have equalities for $(\a-1,\b-1,\c-1,k-1)$ and $(\a-1,\b-1,\c-1,k-2)$. Then a similar argument will claim equalities for all $(\a-i,\b-i,\c-i,k-i)$ and $(\a-i,\b-i,\c-i,k-i-1)$ for all $i\geq 0$ (the first pair of equalities being the ones of the hypothesis).

In particular, we have equality for $k-i=1$, so 
$(\a-(k-1),\b-(k-1),\c-(k-1),1)$ and $(\a-(k-1),\b-(k-1),\c-(k-1),0)$. Now, this implies that we have also all the equalities:
\begin{center}
$(\a-(k-1)+j,\b-(k-1)+j,\c-(k-1)+j,1)$ and for all $j\geq0$.
\end{center}
Now we use $(\a-(k-1)+1,\b-(k-1)+1,\c-(k-1)+1,1)$ and the equality for $(\a-i,\b-i,\c-i,k-i)$ with $i=k-2$ to obtain an equality for $(\a-(k-2)+1,\b-(k-2)+1,\c-(k-2)+1,2)$; then we use the equality for 
$(\a-(k-2)+1,\b-(k-2)+1,\c-(k-2)+1,2)$ and  $(\a-(k-1)+j,\b-(k-1)+j,\c-(k-1)+j,1)$ with $j=2$ to obtain an equality
$(\a-(k-2)+2,\b-(k-2)+2,\c-(k-2)+2,2)$. We repeat this procedure for $j=3,4,\ldots$ to obtain all the equalities claimed in \eqref{eq:a-j}.

Let us now show \eqref{eq:numerical_colex_uniq} using an inductive argument.
The result holds true for $m<k$ and all $k>1$ as one can check. Assume the result holds for some $k'<k$ and $m'<m$. First we use the same argument as in Claim~\ref{cl.3} to conclude that, if we want equality 
\begin{equation} \label{eq.hip3}
{\a \choose k-1}={\b \choose k-1}+{\c \choose k-2}
\end{equation}
then we can assume that $\c$ is a $k-1$ binomial decomposition and that $\b$ is a $k$-binomial decomposition with $\b\geq_{\text{lex}}\a-1$. The assumption that $\b$ is a $k$-binomial decomposition only excludes, under the premise that
${\a \choose k-1}={\b \choose k-1}+{\c \choose k-2}$, the case where $\b=\a-1=\c$ which is a valid decomposition and it is covered by \eqref{eq:numerical_colex_uniq} with $i=t+1$.
Thus, we may assume that $\a,\b,\c$ are binomial decompositions,
 and we proceed as the decomposition in the proof of Lemma~\ref{lem:abck}.
So, let $r$ be the largest index for which $a_r-(k-r)=a_0-k>0$ (the $>0$ holding by the assumption $m\geq k$). Let 
$\a'=(a_0,a_1,\ldots,a_{r-1},a_r-1,a_{r+1},\ldots,a_t)$ which still have $t<k-1$, and is a $k$-binomial decomposition as $m\geq k$.
Then we obtain a decomposition (using Corollary~\ref{cor:rec_const_fin} and the arguments in the final steps of the proof of Lemma~\ref{lem:abck} to remove the wall to the empty set):
\begin{align}
{\a\choose k}&\doteq {\a'\choose k}+{a_{r}-1 \choose k-r-1} \doteq {\a'\choose k}+S+R \nonumber \\
{\b\choose k}+{\c\choose k}&\doteq {\b'\choose k}+{\c' \choose k-1}+S+C \nonumber \\
{\a'\choose k}+S+R&={\b'\choose k}+{\c' \choose k-1}+S+C \label{eq:oo}
\end{align}
with $S$ different from the empty set.
Since $\b$ and $\c$ are binomial decompositions, then we use Observation~\ref{obs:conseq_of_proof} to conclude that $C$ is the empty set (for otherwise we conclude that the hypothesis \eqref{eq.hip3} does not hold).
Thus, we have that \eqref{eq:oo} reads as
\begin{equation}
{\a'\choose k}+S+R={\b'\choose k}+{\c' \choose k}+S
\end{equation}
Now we know that:
\begin{equation}
{\a'\choose k-1}+S'+R'={\b'-\choose k}+{\c'-1 \choose k-2}+S'
\end{equation}
as ${\a \choose k-1}={\b \choose k-1}+{\c \choose k-2}$
and thus
\begin{equation}
{\a'\choose k-1}={\b'\choose k}+{\c'-1 \choose k-2}
\end{equation}
since $R$ is just a bunch of binomial coefficients of the type ${x\choose 0}$ with $x \geq 0$ being an integer, then $R'=\sum {x \choose -1}=0$. However, by letting $\a''$ such that ${\a''\choose k}\stackrel{b}{=}{\a'\choose k}+R$, if $R$ is not the empty set then ${\a'\choose k}+R>{\a'\choose k}$, and since $\a'$ has $t+1<k$ coefficients, then 
${\a''\choose k-1}>{\a'\choose k-1}$, as ${\a'''\choose k}\stackrel{b}{=}{\a'\choose k}+1$ is such that $a_{t+1}'''=k-(t+1)\geq 1$, which means that ${\a'''\choose k-1}>{\a'\choose k-1}$, and since $\a''\geq_{\text{lex}}\a'''$, then ${\a''\choose k-1}\geq{\a'''\choose k-1}> {\a'\choose k-1}$ as claimed.

Therefore, we have that $R$ is the empty set, and thus
\begin{equation}
{\a'\choose k}+S={\b'\choose k}+{\c' \choose k-1}+S,\qquad {\a'\choose k-1}+S'={\b'\choose k-1}+{\c' \choose k-1-1}+S'
\end{equation}
and so 
\begin{equation}
{\a'\choose k}={\b'\choose k}+{\c' \choose k-1},\qquad {\a'\choose k-1}={\b'\choose k-1}+{\c' \choose k-1-1}
\end{equation}
which, by induction hypothesis and since $\a'$ has $t$ coefficients, then \eqref{eq:numerical_colex_uniq} is also satisfied. Indeed, observe that
\begin{align}
{\a'\choose k}&\doteq{\b'\choose k}+{\c' \choose k-1}\qquad \text{ by induction hypothesis}\nonumber \\
{\a \choose k}&\doteq{\a'\choose k}+S, S={a_{r}-1\choose k-r-1}, \qquad \text{ by construction of $S$ and argument on $R=\emptyset$}\nonumber \\
{\b\choose k}+{\c\choose k-1}&\doteq {\b'\choose k}+{\c'\choose k-1}+S \qquad \text{ by construction of $S$ and argument on $D=\emptyset$}\nonumber \\
&\text{thus (translation invariance is transitive)}\nonumber\\
{\a\choose k}&\doteq{\b\choose k}+{\c \choose k-1}
\end{align}
which implies \eqref{eq:conseq1}. The part \eqref{eq:numerical_colex_uniq} follows by using the induction hypothesis on $(\a',\b',\c')$, and then the fact that $S$ should be produced from ${\b\choose k}+{\c\choose k-1}$ using the operations described in Proposition~\ref{prop:char_ti} (in fact, depending on $\a$, $\b'$ and $\c'$, the choices for $\b$ and $\c$ are very limited, for instance, if the term ${a_{r}-1\choose k-r}$ is split into ${a_{r}-2\choose k-r}+{a_{r}-2\choose k-r-1}$,  between $\b'$ and $\c'$, and $a_{r+1}\leq a_r-3$, then $\b$ and $\c$ would split the term into two ${a_{r}-1\choose k-r}+{a_{r}-1\choose k-r-1}$ as well); one may perform a small case analysis together with the fact that the term ${a_{r}-1 \choose k-r-1}$ should be spared among $\b'$ and $\c'$ without leaving any rubble. The proof is completed.

\section{Unicity of the colex order}\label{sec:ucolex}

In this section we prove Theorem \ref{thm:ucolex} on the unicity of the colex order. We start with part \ref{en:ucol_1} of the Theorem. We call a sequence noncomplete (at level $k$) if $\ell (\a)<k$.

\begin{lemma}\label{lem:l(a)<k}
Let  $S$ be an extremal family of $\binom{[n]}{k}$ and $\binom{\a}{k}\stackrel{b}{=}|S|$. If $\a$ is noncomplete at level $k$, then $S$ is  the initial segment of length $m=|S|$ in  the colex order.
\end{lemma} 

\begin{proof}
The proof is by induction on $m=|S|$ and $k$. We recall that, without loss of generality, we assume that $[n]$ is the support of $S$. Since $S$ is an extremal set, the support is $\Delta^{k-1}(S)$ with cardinality $a_0$ if $\ell(\a)=1$ or $a_0+1$ if $\ell (\a)>1$. 

Choose an element $i\in [n]$ with minimum degree and let   $\binom{\b}{k}\stackrel{b}{=}|S\setminus S(i)|$ and $\binom{\c}{k-1}\stackrel{b}{=}|S(i)|$.
For $m\le k$ the result holds trivially as  all families of $k$--sets of $[k+1]$ are isomorphic to initial segments of the colex order. 
Suppose $m>k$. Then $a_0>k$ and by Theorem~\ref{thm:charac} $\b$ is not empty and by Lemma~\ref{lem:a-1<b}, $\b\geq_{\text{lex}}\a-1$.

If $\b=_{\text{lex}}\a-1$ then Theorem~\ref{thm:charac} implies that $S(i)\setminus i$ is extremal and $\Delta(S\setminus S(i))\subset S(i)\setminus i$. Since $\b=_{\text{lex}}\a-1$ we have $\c=_{\text{lex}}\a-1$ as well, namely, $|S(i)|={\a-1\choose k-1}$. By  Theorem~\ref{thm:kk} we have $\Delta(S\setminus S(i))\ge {\a-1\choose k-1}$ and the inclusion $\Delta(S\setminus S(i))\subset S(i)\setminus i$ implies that equality holds. It follows that $S\setminus S(i)$ is an extremal set as well. By induction, as ${\b\choose k}={\a-1\choose k}$ is noncomplete, $S\setminus S(i)$ is an initial segment in the colex order. Therefore $\Delta (S\setminus S(i))=S(i)\setminus i$ is also an initial segment in the colex order. By choosing $i=1$, $S(i)$ is also an initial segment in the colex order. It follows that $S=(S\setminus S(i))\cup S(i)$ is also an initial segment in the colex order. This completes this case.

Suppose $\b>_{\text{lex}}\a-1$. Then Theorem~\ref{thm:charac} implies that both $S(i)\setminus i$ and $S\setminus S(i)$ are extremal and $S(i)\setminus i \subset \Delta(S\setminus S(i))$. By \eqref{eq:numerical_colex_uniq}, both ${\b\choose k}$ and ${\c\choose k-1}$ are incomplete.
By induction both sets are initial segments in the colex order. It remains to show that there is a common ordering of the elements of the support in order that their union is also an initial segment in the colex order. 

We observe that, in the initial segment of the colex order $I$ with support $[n]$, then we have $d_I(1)\ge d_I(2)\ldots \ge d_I(n)$.  Reciprocally, if $d_I(i_1)\ge d_I(i_2)\ge \cdots \ge d_I(i_n)$ then $I$ is isomorphic to the initial segment in the colex order (the isomorphism is given by the ordering $i_1<i_2<\cdots <i_n$ of the support). In oher words, two elements with the same degree can be exchanged while preserving the property that $I$ is an initial segment in the colex order.   Therefore the Claim below completes the proof of the Theorem.     

Given $i$ in the support $[n]$ of $S$, let us use the notation $A_i=S(i)\setminus i$, $B_i=\Delta (S\setminus S(i))$ and $C_i=[n]\setminus \{i\}$.

\begin{claim} Suppose $\b>_{\text{lex}}\a-1$. There is an ordering $x_1<x_2<\cdots <x_{n-1}$ of the elements of $C_i$ such that $d_{A_i}(x_1)\le \cdots \le d_{A_i}(x_{n-1})$ and $d_{B_i}(x_1)\le \cdots \le d_{B_i}(x_{n-1})$.
\end{claim}
\begin{proof} As just discussed (using \eqref{eq:numerical_colex_uniq} and the induction hypothesis) if $\b>_{\text{lex}}\a-1$, both $A_i$ and $B_i$ are initial segments in the colex order of $[n]\setminus \{i\}$. We consider three possible cases in \eqref{eq:numerical_colex_uniq}:
\begin{description}
\item[Case 1.] The splitting is as in \eqref{eq:numerical_colex_uniq}, first equation, and $i=0$ ($i$ in \eqref{eq:numerical_colex_uniq}). Then $|S\setminus S(i)|={a_0\choose k}$, $|\Delta (S\setminus S(i))|={a_0\choose k-1}$ and $\Delta (S\setminus S(i))$ is the initial segment isomorphic to ${[a_0]\choose k}$, in which all elements have the same degree. Therefore all orderings of elements of its support are eligible for an initial segment in the colex order, in particular the one ordering for which $S(i)\setminus i$ is an initial segment in the colex order.
\item[Case 2.] The splitting is as the second equation in \eqref{eq:numerical_colex_uniq}. Since $\b>_{\text{lex}}\a-1$, then $i<t+1$ ($i$ in \eqref{eq:numerical_colex_uniq}). In this case $|\Delta (S\setminus S(i))|=\sum_{i=0}^j{a_i-1\choose k-i-1}+\sum_{i=j+1}^t{a_i\choose k-i-1}$ and $|S(i)\setminus i|=\sum_{i=0}^j{a_i-1\choose k-i-1}$ for $j<t+1$. Since both are initial segments, there is one element with maximum degree ${a_0-1\choose k-2}$ in both sets. Since $S(i)\setminus i \subset \Delta(S\setminus S(i))$, this should be the same element. The next largest degree is common to  $(a_0-1)-(a_1-1+1)+1$ elements with degree
$$
\binom{a_0-2}{k-1}+\binom{a_1-1}{k-1-1}+\cdots+\binom{a_{j}-1}{k-j-1}+\binom{a_{j+1}}{k-(j+1)-1}+\cdots+\binom{a_t}{k-t-1}.
$$
Again the condition $S(i)\setminus i \subset \Delta(S\setminus S(i))$ implies that both sets of $(a_0-1)-(a_1-1+1)+1$ elements are the same. We proceed in a similar way up to the $(j-1)$--th largest degree. For the remaining points in the support, they have the same degree in $A_i$ and at least this degree in $B_i$, so every ordering of these elements in $B_i$ is suitable for the ones in $A_i$.
\item[Case 3.]  If the splitting is as the first equation in \eqref{eq:numerical_colex_uniq} with $i>0$. A similar argument as in the Case 2 yet exchanging the roles of $A_i$ and $B_i$ applies.
\end{description}
\end{proof}
This finishes the proof of the statement.
\end{proof}

The following example show that, if $\ell (\a)=k$ and $m\neq {n'\choose k}-1$, then there are extremal families different from the initial segment of the colex order. Let
$$
m=|S|\stackrel{b}{=} {a_0\choose k}+{a_1\choose k-1}+\cdots +{a_{k-1}\choose 1}, a_{k-1}\ge 1.
$$
The initial segment of length $m$ in the colex order is
$$
I_{n,k}(m)={[a_0]\choose k}\cup \left(\{a_0+1\}\vee {[a_1]\choose k-1}\right)\cup\cdots \cup \left(\{a_0+1,a_1+1,\cdots ,a_{k-2}+1\}\vee {a_{k-1}\choose 1}\right).
$$
Let $r$ be the smallest index with the property that 
$a_r-(k-r)=a_{k-1}-1$, that is, the binomial coefficient ${a_r\choose k-r}$ is in the same diagonal as ${a_{k-1}\choose 1}$. Assume $r>0$ (for otherwise $m=\binom{n'}{k}-1$). Take the set  $X=\{a_0+1,a_1+1,\ldots ,a_{k-2}+1,a_{k-1}\}\in I_{n,k}(m)$ and replace $a_r+1$ by $a_{k-1}+1$. Let $X'$ be the resulting set. Then 
$$
S=(I_{n,k}(m)\setminus \{X\})\cup \{X'\},
$$
has cardinality $m$, it has the same shadow as $I_{n,k}(m)$ and it is not isomorphic to the initial segment of the colex order. If $r=0$ then $m={a_0+1\choose k}-1$ and the support of $S$ is $[a_0]$. In this case, by the last part of   Theorem~\ref{thm:kk+fg}, the only examples of extremal sets are of the form ${[a_0+1]\choose k}\setminus \{x\}$ for some $x$, and all these are isomorphic to the initial segment of the colex order. This completes the proof of Theorem~\ref{thm:ucolex}.

\section{Final remarks}\label{sec:concl}

In this paper we have focussed in the arithmetic nature of the Kruskal--Katona theorem. The key inequality in Lemma \ref{lem:abc} has interest in its own. Even if binomial decompositions appear mostly in connection to the Kruskal--Katona theorem, they are also related to the study of  simplicial complexes  and combinatorial geometry (the initial motivation in the Kruskal paper), and  in isoperimetric problems in the $n$--cube, where the characterization of extremal sets is still an open problem (see e.g. Bezrukov \cite[section 4]{bezrukov1994isoperimetric}). We think that the role of the arithmetic inequality \eqref{eq:abc0} in the shadow optimization problem is remarkable and gives perhaps the more natural way to prove the Kruskal--Katona theorem. In particular, it easily leads to  the characterization of the extremal families given in Theorem \ref{thm:charac}. We note that the key inequality in Lemma \ref{lem:abck} for $k',k''\ge k$ can be obtained from the Kruskal--Katona theorem, but we have not found a way to derive the significant case $k'=k, k''=k-1$ from it. 

In most relevant applications the simpler form by Lov\'asz of the Kruskal--Katona theorem is usually enough. Lov\'asz formulation states that, if $x$ is the real number defined by $m={x\choose k}$, then every family $S$ of $k$--sets with cardinality $m=|S|$ has lower shadow of cardinality
$$
|\Delta (S)|\ge {x\choose k-1},
$$
For the above inequality, the only cases of equality are the families ${[n]\choose k}$. It may be of interest to investigate the analogous version of Lemma \ref{lem:abck} for this case, which we conjecture to hold.

\begin{conjecture} Let $x,y,z\in \R$ such that ${x\choose k}={y\choose k}+{z\choose k-1}$ with $y\ge x-1$. Then,
$$
{x\choose k-1}\le {y\choose k-1}+{z\choose k-2}.
$$
\end{conjecture}

The analogous statement replacing $k$ and $k-1$ in the right--hand side of the first equality by $k',k''\ge k$ indeed holds and it is simple to verify.

 Further insight on the structural properties of the extremal families can be obtained by an additional approach which we develop in a separate paper (see \cite{Serra2021752} for an extended abstract and  also  \cite{Serra20191043}). For an extremal family  $S$ (actually any family of $k$--sets), $\Delta^{k-1}(S)$ is trivially an initial segment of the colex order.  One possible   measure of the distance between extremal sets and initial segments of the colex order is the minimum $r$ such that $\Delta^{r}(S)$ is an initial segment of the colex order ($r=0$ if $S$ is an initial segment itself, and extremal families with $r=1$ can be seen as small perturbations of initial segments). We show in \cite{Serra2021752} that we actually have  $r=O(loglog (n))$, which shows that extremal sets are not far from initial segments in that sense. Moreover, large values of $r$ involve a fast decreasing (doubly exponential)  of the sequence of   coefficients in the binomial decomposition of $m=|S|$,  thus making  such extremal sets rather rare. We also provide an explicit algorithm polynomial in $n$ to construct extremal families which are at a given distance to  initial segments of the colex order in the above sense. The algorithm provides examples of extremal sets which are structurally further away from initial segments than the simple examples given at the end of Section \ref{sec:ucolex}.

Our original motivation of the current paper is the isoperimetric problem in Johnson graphs (see e.g. \cite{bollobas2004isoperimetric, keevash2008shadows, Serra20191043}). The Johnson graph $J(n,k)$ has vertex set ${[n]\choose k}$ and the set of neighbours of a set $S$ is $\Delta(\nabla(S))$. The problem asks to minimize the boundary of sets $S$ of vertices in  $J(n,k)$ for each given  cardinality of $S$. A family of solutions to the problem is not known and, unlike closely related problems of this kind, a family of nested solutions is known not to exist. Extremal families for the Kruskal--Katona theorem provide asymptotic examples of extremal sets for the isoperimetric problem in the Johnson graphs. It turns out that there are extremal families different from initial segments in the colex order that beat the initial segments in such asymptotic constructions. Therefore, a better understanding of the extremal families for the Kruskal--Katona theorem is instrumental in this isoperimetric  problem as well.


\bibliographystyle{abbrv}
\bibliography{bib_john.bib}


\end{document}